\documentclass[11 pt]{article}
\usepackage{subfiles}
\usepackage[backref=page]{hyperref}
\usepackage[alphabetic]{amsrefs}

\usepackage[all,2cell]{xy} \UseAllTwocells \SilentMatrices \SelectTips{cm}{}
\usepackage{latexsym,amsfonts,amssymb}
\usepackage{amsmath,amsthm,amscd}

\usepackage{color}
\usepackage{etoolbox}
\usepackage{graphicx}
\usepackage{a4wide}
\usepackage{fancyhdr}
\usepackage{mathtools}
\usepackage{topcapt}
\usepackage{booktabs}
\usepackage{comment}
\usepackage[scaled]{helvet}
 \usepackage[colorinlistoftodos]{todonotes}
\usepackage[T1]{fontenc}
\input xy
\xyoption{all}

\usepackage{geometry}

\usepackage{todonotes,setspace}

\newcommand{\ot}{\otimes}
\newcommand{\R}{\mathbb{R}}
\newcommand{\Z}{\mathbb{Z}}

\newcommand{\C}{\mathbb{C}}

\newcommand{\A}{\mathbb{A}}
\renewcommand{\P}{\mathbb{P}}
\newcommand{\inv}{^{-1}}
\newcommand{\Hom}{\text{\rm Hom}}
\newcommand{\Mat}{\text{\rm Mat}}

\newcommand{\PSL}{\text{\rm PSL}}
\newcommand{\SL}{\text{\rm SL}}

\renewcommand{\O}{\mathcal{O}}
\newcommand{\cF}{\mathcal{F}}

\newcommand{\cO}{\mathcal{O}}
\newcommand{\cM}{\mathcal{M}}

\newcommand{\U}{\mathcal{U}}
\newcommand{\cI}{\mathcal{I}}

\newcommand{\QCoh}{\mathsf{QCoh}}

\newcommand{\Spec}{\text{\rm Spec}}

\newcommand{\Sym}{\text{\rm Sym}}

\newcommand{\dmod}{\text{\rm -mod}}

\newcommand{\End}{\text{\rm End}}

\newcommand{\Loc}{\mathsf{Loc}}

\newcommand{\D}{\mathcal{D}}
\newcommand{\Y}{\mathcal{Y}}
\newcommand{\X}{\mathcal{X}}

\newcommand{\Asymp}{\text{\rm Asymp}}
\newcommand{\horocycle}{\frac{G/N \times N^- \setminus G}{H}}

\newcommand{\filt}{^{\text{\rm filt}}}

\newcommand{\g}{\mathfrak{g}}
\renewcommand{\k}{\mathfrak{k}}

\newcommand{\frakn}{\mathfrak{n}}

\newcommand{\frakk}{\mathfrak{k}}
\newcommand{\fraku}{\mathfrak{u}}
\newcommand{\frakl}{\mathfrak{l}}
\newcommand{\fraks}{\mathfrak{s}}
\newcommand{\frakp}{\mathfrak{p}}
\newcommand{\fg}{\mathfrak{g}}
\newcommand{\fa}{\mathfrak{a}}
\newcommand{\fb}{\mathfrak{b}}
\newcommand{\fn}{\mathfrak{n}}
\newcommand{\fm}{\mathfrak{m}}
\newcommand{\fh}{\mathfrak{h}}

\newcommand{\fk}{\mathfrak{k}}

\newcommand{\fl}{\mathfrak{l}}
\newcommand{\fZ}{\mathfrak{Z}}
\newcommand{\Lie}{\text{\rm Lie}}
\newcommand{\Ug}{U\g}

\newcommand{\Vinb}{\mathbb{V}}
\newcommand{\VinbX}{\mathbb{V}_X}
\newcommand{\VXcirc}{{\mathbb{V}}^\circ_X}
\newcommand{\VinbG}{\mathbb{V}_G}

\newcommand{\Gad}{G^\text{\rm ad}}
\newcommand{\Gbarad}{\overline{G^\text{\rm ad}}}
\newcommand{\Gbar}{\overline{G}}
\newcommand{\gr}{\text{\rm gr}}
\newcommand{\Rees}{\text{\rm Rees}}

\newcommand{\Lad}{L_I^{\text{\rm ad}}}

\newcommand{\Usl}{U(\mathfrak{sl}_2)}

\newcommand{\Derv}{\text{\rm Derv}}

\renewcommand{\frakk}{\mathfrak k}
\newcommand{\ad}{\text{\rm ad}}
\newcommand{\submon}{_{\text{\rm mon}}}

\newcommand{\wt}{\widetilde}

\numberwithin{equation}{section} 

\theoremstyle{plain}
\newtheorem{theorem}[equation]{Theorem}
\newtheorem{prop}[equation]{Proposition}
\newtheorem{defprop}[equation]{Definition-Proposition}
\newtheorem{lemma}[equation]{Lemma}
\newtheorem{cor}[equation]{Corollary}

\theoremstyle{definition}
\newtheorem{definition}[equation]{Definition}
\newtheorem{rmk}[equation]{Remark}
 
\newtheorem{example}[equation]{Example}

\title{Wonderful asymptotics of matrix coefficient $D$-modules}
\author{David Ben-Zvi and Iordan Ganev}
\date{}

\begin{document} 

\maketitle

\begin{abstract}
Beilinson--Bernstein localization realizes representations of complex reductive Lie algebras as monodromic $D$-modules on the ``basic affine space'' $G/N$, a torus bundle over the flag variety. A doubled version of the same space appears as the horocycle space describing the geometry of the reductive group $G$ at infinity, near the closed stratum of the wonderful compactification $\Gbar$, or equivalently in the special fiber of the Vinberg semigroup of $G$. We show that Beilinson--Bernstein localization for $U\fg$-bimodules arises naturally as the specialization at infinity in $\Gbar$ of the $D$-modules on $G$ describing matrix coefficients of Lie algebra representations. More generally, the asymptotics of matrix coefficient $D$-modules along any stratum of $\Gbar$ are given by the matrix coefficient $D$-modules for parabolic restrictions. This provides a simple algebraic derivation of the relation between growth of matrix coefficients of admissible representations and $\mathfrak n$-homology. The result is an elementary consequence of the compatibility of localization with the degeneration of affine $G$-varieties to their asymptotic cones; analogous results hold for the asymptotics of the equations describing spherical functions on symmetric spaces. 
\end{abstract}

\tableofcontents

\section{Introduction}\label{sec:intro}

Our goal is to explore a simple relation between one of the main themes in the harmonic analysis on real reductive groups, the Harish-Chandra theory of asymptotics of matrix coefficients, and one of the main tools in geometric representation theory, the Beilinson--Bernstein localization of representations on flag varieties. The relation is mediated by the wonderful compactification or, equivalently, the Vinberg degeneration of $G$. The perspective of this paper was inspired by the seminal works of Bezrukavnikov--Kazhdan~\cite{BezrukavnikovGeometrysecondadjointness2015} and Sakellaridis--Venkatesh~\cite{SakellaridisPeriodsHarmonicAnalysis2017} which utilize a geometric interpretation of asymptotics and scattering in the $p$-adic setting.

\subsection{Beilinson--Bernstein localization and matrix coefficients}\label{sec:intro:bb}
Let us fix a complex, connected reductive group $G$ with Borel $B$, unipotent radical $N = R_{\rm u }(B)$, and Cartan $H=B/N$, and denote the corresponding Lie algebras by $\fg,\fb,\fn$ and $\fh$. The Beilinson--Bernstein localization theorem~\cite{BEILINSONLocalisationgmodules1981} (see also \cite{TakeuchiDModulesPerverseSheaves2007} for an exposition) is a generalization of the Borel--Weil--Bott realization of representations of $G$ through the cohomology of line bundles on the flag variety. Localization defines a functor
$$\Loc_{G/B,\lambda}:U\fg\dmod_{[\lambda]}\to D_\lambda(G/B)$$ from representations of the enveloping algebra $U\fg$ with fixed infinitesimal character $[\lambda]\in \fh^*/W$ to $D$-modules on $G/B$ twisted by $\lambda\in \fh^*$. At a point $x$ in the flag variety with stabilizer $\fb_x$ and associated nilpotent radical $\fn_x$, the functor $\Loc_{G/B,\lambda}$ calculates the $\lambda$-isotypic part of the $\fh$-action on $\fn_x$-coinvariants (zeroth $\fn_x$-homology) of a representation. The localization theorem asserts that $\Loc_{G/B,\lambda}$ is an equivalence of abelian categories for $\lambda$ dominant and regular. 

The construction extends naturally to describe arbitrary representations of $\fg$ by replacing $G/B$ by the ``basic affine'' space $G/N$, which forms an $H$-bundle over the flag variety $G/B$, and considering the category $D_H(G/N)$ of  $H$-monodromic $D$-modules\footnote{An $H$-monodromic structure on a $D$-module is an $H$-equivariant structure on the underlying quasi-coherent sheaf, compatible with the action of the sheaf of differential operators. Such $D$-modules are also known as weakly $H$-equivariant, conical, or homogeneous.} on $G/N$. The resulting localization functor $$\Loc_{G/N}:U\fg\dmod\longrightarrow D_H(G/N)$$ attaches to the fiber in $G/N$ over $x \in G/B$ the full $\fh$-module given by the $\fn_x$-homology of a representation, considered as a monodromic $D$-module on $H$. This functor is studied in~\cite{Ben-ZviBeilinsonBernsteinlocalizationHarishChandra2012a}, where its derived version is shown to be monadic and comonadic, providing a derived equivalence between representations and modules (or comodules) over a form of the Demazure Hecke algebra acting on $D_H(G/N)$. (Note that in the current paper we work for concreteness solely on the abelian level.) 

For any variety $X$ with an action of an algebraic group $G'$, there is a localization functor $\Loc_X : U\fg'\dmod\to  D(X)$ from representations of the Lie algebra $\g'$ to $D$-modules on $X$. This functor originates from the action map (or quantum moment map) $U\fg'\to \Gamma(X,\D_X)$ valued in the algebra of global differential operators on $X$, and is left adjoint to the global sections functor $\Gamma: D(X)\to U\fg'\dmod$. We are interested in the instance of the localization construction that takes place on the group $G$ itself, considered as the symmetric space $X=G'/K$ where $G'=G\times G$ and $K=G$ (i.e., the action of $G\times G$ on $G$ by left and right multiplication). In this case, we call the resulting functor 
$$\Loc_G:U\fg\ot U\fg\dmod \longrightarrow D(G)$$ 
the {\em matrix coefficients} functor. Fiberwise, this functor calculates coinvariants of representations for the $\fg\oplus \fg$-stabilizers of points in $G$. 
The matrix coefficients functor has a universal origin, in the tautological relation between representations of any group $G$ and functions on the group itself given by matrix elements of the group action 
\begin{equation} \label{eq:mc} V \ot V^*\to \text{\rm Fun}(G),\hskip.3in v \ot v'\mapsto \left[ m_{v,v'} : g \mapsto \langle v', g\cdot v\rangle \right].\end{equation} 
Here $V$ is any representation of $G$ and $\text{\rm Fun}(G)$ is an appropriate space of functions on $G$ (depending on the type of representation).  Concretely, the matrix coefficients functor translates the $U\fg$-relations satisfied by differentiable vectors $v$ and $v'$ into a system of differential equations on $G$ imposed on the corresponding matrix coefficient functions\footnote{The (derived) matrix coefficients functor is recovered as matrix coefficients for $U\fg\dmod$ considered as a strong categorical representation of $G$ (i.e., module category for the monoidal dg category $(D(G), \ast)$ of $D$-modules on $G$). Taking matrix coefficients with the ``spherical vector'' $\text{\rm Ind}_{\fk}^{\fg}\C\in (U\fg\dmod)^K$ we recover the localization of $U\fg\dmod$ on the symmetric space $G/K$ as a categorical form of spherical functions.} $m_{v,v'}$. See Section~\ref{real groups intro} below for a discussion of our main motivation, the relation with matrix coefficients of admissible representations of real forms of $G$. 


\subsection{Wonderful asymptotics}\label{sec:intro:wonderful}
Harmonic analysis on symmetric spaces provides a different source of intuition for Beilinson--Bernstein localization. Specifically, flag varieties appear as strata in compactifications of symmetric spaces, with associated horocycle spaces appearing as models for the geometry of the space near infinity.  

To be more specific, let $Z(G)$ be the center of $G$ and $\Gad = G/Z(G)$ be the adjoint group. Let $H^{\text{\rm ad}} = H/Z(G)$ be the adjoint torus, which is a maximal torus of $\Gad$. We denote by $B^-$ the Borel subgroup of $G$ opposite to $B$, and $N^-$ its unipotent radical. 

Recall that the adjoint group $\Gad$ admits a smooth, projective $G \times G$-equivariant compactification $\Gbarad$, known as the wonderful compactification (See, e.g.\ \cite{ConciniCompactificationsymmetricvarieties1999, Evenswonderfulcompactification2008a}). The $G \times G$ orbits on $\Gbarad$ are indexed by subsets of positive simple roots, and the unique closed orbit is identified with $G/B \times B^-\backslash G$, where the notation makes clear the $G$ actions from the left and right. We record the following consequence of these observations:
\begin{quote}
	The  group $G$ ``looks at infinity'' like the open $G\times G$ orbit normal cone to the closed stratum $G/B\times B^-\backslash G\subset\Gbar$ in the wonderful compactification. This ``deleted normal cone'' is in turn identified with the  {\em horocycle space} for $G$, namely, $$\Y=G/N\times_H N^-\backslash G.$$
\end{quote}
Thus, the horocycle space is formed by taking the balanced product of the right and left $H$-actions on $G/N$ and $N^- \backslash G$, respectively, and forms a $G\times G$-equivariant $H$-bundle over closed stratum $G/B\times B^-\backslash G$.  

As a result, there are natural ways to relate $D$-modules on $G$ to their ``asymptotics'', which are $D$-modules on $\Y$. 
The theory of Verdier specialization and the Kashiwara--Malgrange V-filtration provide a functor from holonomic $D$-modules on $G$ to holonomic $D$-modules on $\Y$. The closures of the $r$ codimension 1 orbits in $\Gbarad$ constitute the boundary divisors, and these give rise to a multi-version of the $V$-filtration on the sheaf $\D_{\Gbarad}$ of differential operators on $\Gbarad$ (as well as on $j_* \D_{\Gad}$, where $j : \Gad \rightarrow \Gbarad$ is the inclusion).  

On the other hand, by the Peter--Weyl theorem, the matrix coefficients of finite dimensional representations of $G$ span the algebra $\O(G)$ of polynomial functions on $G$ and, on the level of rings, induce a filtration by the weight lattice $\Lambda$ of $G$. Thus, one considers the category of $\Lambda$-filtered $D$-modules on $G$, where the filtration is required to be compatible with the matrix coefficients filtration on $\O(G)$. Passing to the associated graded defines a functor 
\begin{equation}\label{eq:intro:asymp}\text{\rm Asymp}: D(G)\filt\longrightarrow D_H(\Y)\end{equation}
from the category of filtered $D$-modules to $H$-monodromic $D$-modules on $\Y$. We check in Proposition \ref{prop:vfiltration} that the two pictures are compatible -- the $V$-filtration on $D_{\Gad}= \Gamma(\Gad,  \D_{\Gad})$ coincides with the matrix coefficients filtration on $D_{\Gad}$, so that the Verdier specialization of a holonomic $D$-module coincides with the associated graded for the unique compatible filtration.

Our main observation is that the localization functor naturally lifts to give filtered $D$-modules on $G$, and thus admits an elementary version of asymptotics. In particular, the specialization at infinity of holonomic modules arising by localization is likewise given by localization, and is thus easy to calculate.

\begin{theorem}[Theorem \ref{thm:loc} and Proposition \ref{prop:vfiltration}]\label{thm:intro:loc}
The asymptotics of matrix coefficient $D$-modules is given by doubled Beilinson--Bernstein localization: the natural localization functor $ \Ug\ot\Ug\dmod\to D(G)$ lifts to $D(G)\filt$, and we have a commutative diagram
$$\xymatrix{\Ug\ot\Ug\dmod \ar[rrrrr]^{\qquad  \mbox{\rm matrix coefficients}} \ar[drrrrr]_-{\mbox{\rm Beilinson--Bernstein \quad \qquad}} &&&&& D(G)\filt\ar[d]^-{\mbox{\Asymp}   }\\
&&&&&D_H(\Y)}$$
Moreover, on representations with holonomic localization, the asymptotics functor is naturally identified with Verdier specialization or nearby cycles. \end{theorem}

\begin{rmk}  The infinitesimal action of $\fg\oplus\fg$ on $G$ gives rise to an action which identifies the left and right actions of the center $\fZ\subset U\fg$:
\begin{equation}\label{factoring center}
U\fg\ot U\fg\longrightarrow U\fg\ot_\fZ U\fg\longrightarrow D(G)
\end{equation} Thus, it is natural to change the source of the localization functor correspondingly, as we will do in Section~\ref{scattering}.
\end{rmk}

The commutative diagram above is naturally strongly $G\times G$-equivariant, and so implies similar localization results for various categories of representations such as $(\fg,K)$-modules. For example, in the ``group case'', imposing equivariance for diagonal $G_\Delta \subseteq G \times G$,  we obtain the localization of Harish-Chandra bimodules
$$\xymatrix{\left(\Ug\ot\Ug\dmod\right)^{G_\Delta}=\mathcal{HC} \ar[r] \ar[dr]& D\left(\frac{G}{G}\right)\filt \ar[d]\\
&D_H\left(\frac{N\backslash G/N^-}{H}\right)},$$
where $\frac{G}{G}$ is the adjoint quotient of $G$ by the conjugation action on itself, and $\frac{N\backslash G/N^-}{H}$ is the quotient of the stack $N\backslash G/N^-$ by the diagonal action of $H$. 

\begin{rmk} Our constructions only use those derivations on the wonderful compactification that preserve the $G \times G$ orbits; those that are transverse to the orbits do not appear. In this sense, the full category of $D$-modules on $\Gbarad$ is too large. Instead, a natural setting is logarithmic $D$-modules on $\Gbarad$ (and their twisted variants), using the fact that $\Gbarad$ is a log-homogeneous space in the sense of~\cite{BrionLoghomogeneousvarieties2007a}, in fact a {\em wonderful variety}, in the sense of~\cite{LunaToutevarietemagnifique1996}. The sheaf of log-differential operators is generated by its global sections, which are identified with $U\fg\ot_\fZ U\fg$. This isomorphism is also proved in~\cite{SagatovLogarithmicDifferentialOperators2017} and used to study the logarithmic Beilinson--Bernstein localization on $\Gbarad$. 
\end{rmk}

\subsection{Partial asymptotics and parabolic restriction}\label{parabolic intro}

The picture of asymptotics of matrix coefficient $D$-modules has a natural generalization associated to arbitrary strata of the wonderful compactification, in which the asymptotics functor is identified with parabolic restriction. Namely, the wonderful compactification has a beautiful inductive structure, in which strata are labeled by conjugacy classes of parabolics, the closures of strata fiber over partial flag varieties, and the fibers are themselves identified with wonderful compactifications of Levi subgroups. The asymptotics of matrix coefficients respects this structure, recovering matrix coefficients for parabolic restrictions along the fibers.

To spell this out, recall that the $G\times G$-orbits in the wonderful compactification $\Gbarad$ are in bijection with subsets $I$ of the Dynkin diagram; we fix a subset $I$ of simple roots and denote the corresponding orbit as $\X_I$. The subset $I$ also determines a conjugacy class of parabolic subgroups; let $P_I$ be a representative of this conjugacy class, $N_I$ its unipotent radical, and $L_I$ the corresponding Levi subgroup. The orbit $\X_I$ fibers over the double partial flag variety $G/P_I \times P_I^- \backslash G$ with fiber equal to the adjoint group $M_I$ of the Levi $L_I=M_IA_I$, while the orbit closure $\overline{\X}_I$ fibers over the same base with fiber the wonderful compactification of the adjoint reductive group $M_I$:
$$\xymatrix{\overline{M}_I\ar[r] & \overline{\X}_I\ar[d]\\
	 & G/P_I\times P_I^-\backslash G}$$
Explicitly,  
$$ \X_I\simeq \left(G/A_I N_I \right) \times_{M_I} \left( N_I^-A_I^-\backslash G \right).$$  The deleted normal cone to the stratum $\X_I$ is an $A_I$-torsor over $\X_I$, identified with the `partial horocycle space':
$$\Y_I := G/N_I\times_{L_I} N_I^-\backslash G.$$ In particular, $\Y_I$ is a partial degeneration of $G$, and passing to partial associated graded modules constructions give rise to a functor:
$$\Asymp_I : D(G)\filt \rightarrow D_{A_I}(\Y_I).$$
Moreover, the fibers of $\Y_I\to G/P_I\times P_I^-\backslash G$ are identified with $L_I$ uniquely up to the action of $A_I$.  

We denote the Lie algebras of $P_I$ and $P_I^-$ by $\frakp_I$ and $\frakp_I^-$, with unipotent radicals $\fraku_I$ and $\fraku_I^-$ and common Levi $\frakl_I$.
Passing to coinvariants for $\fraku_I\oplus \fraku_I^-$ defines the parabolic restriction functor
$$U\g\ot U\g\dmod \longrightarrow U\frakl  \ot U\frakl \dmod,$$
which we can then compose with the matrix coefficients functor for the Levi $L$.

\begin{theorem}[See Theorem \ref{thm:parabolic}]\label{thm:intro:parabolic} Partial asymptotics and parabolic restriction are related by the following commutative diagram:
	\[ \xymatrix{  U\g \ot U\g\dmod \ar[rrrrrr]^{\mbox{\rm parabolic restriction}} \ar[d]^-{\mbox{\rm matrix coefficients} } & &  && & &  U\frakl  \ot U\frakl \dmod \ar[d]_-{ \mbox{\rm matrix coefficients}}    \\ 
	D(G)\filt  \ar[rrr]_-{\mbox{\rm partial asymptotics}} &&& D_{A_I}(\Y_I)\ar[rrr]_-{\mbox{\rm restrict to fiber}}  & & & D_{A} (L) }\] 
\end{theorem}

Again, in the case of holonomic modules on $G$ given by localization, we find that the specialization along any stratum is given in elementary algebraic terms.

\subsection{Application: Asymptotics of (actual) matrix coefficients}\label{real groups intro}

In Section~\ref{real groups section} we describe the relation between our construction and the analytic theory of asymptotics of matrix coefficients of admissible representations~\cite{Harish-ChandraDifferentialequationssemisimple1984, WarnerHarmonicAnalysisSemiSimple1972, CasselmanAsymptoticbehaviormatrix1982,MilicicAsymptoticbehaviourmatrix1977,HechtasymptoticsHarishChandramodules1983, HechtCharactersasymptoticsnhomology1983}. Specifically, we explain how the identification of the specialization at infinity of matrix coefficients $D$-modules gives a simple algebraic derivation of the asymptotic description of matrix coefficients, in particular results in~\cite{CasselmanAsymptoticbehaviormatrix1982}, which are themselves modernizations and improvements of classical results of Harish-Chandra~\cite{Harish-ChandraDifferentialequationssemisimple1984}, ~\cite[Vol. II, 9.1.1.1]{WarnerHarmonicAnalysisSemiSimple1972}. Namely, we can read off all the exponents and powers that can appear in the asymptotic expansion of matrix coefficients as functions on a chamber in $\fa_\R$ from the weights and multiplicities of the $\fa$-action on the $\fn$-coinvariants of a representation (for the Iwasawa $\fn$), i.e., from the Beilinson--Bernstein localization at the corresponding point of the flag variety (see Theorem~\ref{asymptotics theorem}).
Parallel results hold for the asymptotics of spherical functions, substituting the asymptotic degeneration of $G/K$ to $G/MN$ (see Section~\ref{scattering}) for the degeneration of $G$ to the horocycle space.

\subsection{Wonderful scattering}\label{scattering}

The static relation between the group and its wonderful compactification has a (multi-temporal) dynamic counterpart given by the Vinberg degeneration of $G$ to its asymptotic cone. On the level of $D$-modules, this relation is given by the Rees construction, which makes Theorem~\ref{thm:intro:loc} and its generalizations evident, as we describe below.
This mirrors the classical relation of asymptotics of eigenfunctions on symmetric spaces with scattering for the wave equation; see Section~\ref{wave section} for the relation of localization with the multi-temporal wave equation of Semenov-Tian-Shansky~\cite{Semenov-Tian-ShanskyHarmonicanalysisRiemannian1976}.

First let us explain a general tautology: localization commutes with passing to graded modules. Suppose we are given a filtered ring $R=\bigcup_i R_i$ and a homomorphism $\mu:A\to R_0$ to the zeroth filtered piece. Then the induction functor $ - \ot_A R: A\dmod\to R\dmod$ naturally factors through filtered $R$-modules, or to graded modules over the Rees construction $\Rees(R)$. Moreover we have a functorial identification
$$\gr(-\ot_A R)\simeq -\ot_A \gr(R)$$
between the associated graded of an induced module and the induction to the associated graded ring. 

We apply this to the case where  $R$ is the coordinate ring of an affine $G$-variety $X$ and $A=\Ug$. The decomposition of $R$ as a direct sum of representations of $G$ does not respect the ring structure in general, but defines a canonical multifiltration of $R$ as a ring indexed by the cone of dominant weights of $G$. The associated Rees construction gives the asymptotic degeneration of $X$, as described (with slight variants) in~\cite{PopovContractionsactionsreductive1987},~\cite[Section 5.1]{GaitsgorySphericalVarietiesLanglands2010}, and~\cite[Section 2.5]{SakellaridisPeriodsHarmonicAnalysis2017}.  Namely, we obtain an $H$-equivariant flat family
$$\pi: \wt{X} \rightarrow \A^r$$ 
of $G$-varieties over $\A^r$, where $r=\dim(H)$, and we regard $\A^r$ as the partial compactification\footnote{As a toric variety for $H^{\text{\rm ad}}$, the affine space $\A^r$ is associated to the cone of dominant weights.} of $H^{\text{\rm ad}}$. The fiber over $1\in H^{\text{\rm ad}} \subset \A^r$ is identified with $X$ itself, while the fiber at $0$ is the {\em asymptotic cone} $\gr(X)$ of $X$ studied in~\cite{PopovContractionsactionsreductive1987}. As shown in~\cite{PopovContractionsactionsreductive1987}, $\gr(X)$ is a {\em horospherical} variety, i.e., the stabilizer of any point contains the unipotent radical of a Borel.  We thus find the following principle (see Section \ref{sec:loccommute}):

\begin{prop}
Let $X$ be an affine $G$-variety with its canonical multifiltration. Localization on $X$ lifts to relative $D$-modules along the degeneration to $\gr(X)$, and specializes to localization on $\gr(X)$: we have a commutative diagram
\[\xymatrix{ &  & \Ug\dmod \ar[d]^{\Loc_{\wt{X}}} \ar[lldd]_{\Loc_{X_0}} \ar[rrdd]^{\Loc_X} &  & \\ & & D_A(\wt{X}/\A^r)  \ar[lld]^{ i_0^* }&   & \\  D_H(X_0) & & & &  D(X)\filt   \ar[llu]^{\text{\rm Rees}}  \ar[llll]^{\text{\rm Asymp } } }  \]
(where we only consider $D$-modules on the smooth loci of the varieties in question).
\end{prop}

Here $D_A(\wt{X}/\A^r)$ denotes the category of $D$-modules on $\wt{X}$ relative to the base $\A^r$, and equipped with an $A$-equivariant structure. The asymptotics functor of the above proposition can be interpreted in terms of Verdier specialization, as we explain in Section \ref{sec:Verdier}. 

To summarize, let $\overline{X}$ be the GIT quotient of $\wt{X}$ with respect to the torus $H$, so that $\overline{X}$ is compactification of $X$ with $r$ irreducible boundary divisors. Let $M$ be a holonomic $D$-module on $X$. Then $M$ has a (multi-)Kashiwara--Malgrange filtration with respect to the boundary divisors. The Verdier specialization of $M$ is the $D$-module obtained by taking the associated graded of $M$, and defines a $D$-module on the normal cone in $\overline{X}$ to the intersection of the irreducible boundary divisors. The space $X_0$ embeds in this normal cone, and so we can restrict the Verdier specialization of $M$ to $X_0$. On the other hand, the Kashiwara--Malgrange filtration is also compatible with the filtration on $\D_X$, so $M$  lifts to an object in $D(X)\filt$ and we can apply the functor $\Asymp$ from the above proposition. The result matches with the restriction of the Verdier specialization of $M$ to $X_0$. 

We can also refine this construction as in~\cite{SakellaridisPeriodsHarmonicAnalysis2017} by replacing the full torus $H$ by its quotient $A_X$ associated to $X$~\cite{KnopLunaVusttheoryspherical1991} and the base by the affine embedding $A_X\hookrightarrow \overline{A_X}$ associated to the cone of $G$-invariant valuations. In the case $X=G$ considered as a $G\times G$-space, we find a filtration of $\C[G]$ by the weight lattice of $G$ (rather than of $G\times G$) and the result is the Vinberg semigroup $\Vinb_G$ \cite{Vinbergreductivealgebraicsemigroups1995}. It is an $H$-equivariant family interpolating between $G$ and its asymptotic cone, the affine closure of the quasi-affine horocycle space, which is embedded as the open $G\times G$-orbit in the fiber over zero:
$$\Y\subset \Vinb_G|_0={\rm Spec}k[\Y].$$
$$\xymatrix{G\ar[d]  \ar[rr]&& \Vinb_G\ar[d]^-{\pi}&&\ar[ll] \overline{\Y}\ar[d]\\
\{1\}\ar[rr]&& \A^r &&\ar[ll] \{0\}}$$
The $H$-orbits on the base $\A^r$ are in bijection with subsets $I$ of the Dynkin diagram. The fiber over a point in the $I$-orbit has a unique open $G\times G$-orbit, which is identified with $\Y_I$. In this way, the Vinberg family realizes also the partial degenerations of $G$ to each $\Y_I$. Localization along this family produces the results of Sections~\ref{sec:intro:wonderful} and~\ref{parabolic intro}.

More generally, let $\theta: G \to G$ be an involution, $K$ the fixed points of $\theta$, and $X=G/K$ the corresponding symmetric space. Let $P$ be a minimal $\theta$-split parabolic subgroup of $G$, and $P = MAN$ its Langlands decomposition. In this case functions on $X$ are naturally multifiltered by the character lattice of $A$ and we find an $A$-equivariant family $\Vinb_{G/K}$ (studied in particular in~\cite{AbeJacquetFunctorConcini2015, ChenformulageometricJacquet2017}) over a base of dimension the rank of $X$, a completion of $A$, degenerating $X$ to the corresponding horocycle space $G/MN$:
$$\xymatrix{G/K\ar[d]  \ar[rr]&& \Vinb_{G/K}\ar[d]^-{\pi}&&\ar[ll] \overline{G/MN}\ar[d]\\
\{1\}\ar[rr]&& \A^{r_K} &&\ar[ll] \{0\}}$$
Localization on this family identifies the asymptotics of the differential equations satisfied by spherical functions with the equations given by Beilinson--Bernstein localization on $G/MN$.

\subsection{Relation to previous work}\label{previous}

This paper is in the spirit of many recent appearances (including~\cite{AbeJacquetFunctorConcini2015, BouthierDimensionfibresSpringer2015, ChenformulageometricJacquet2017, ChenCasselmanJacquetfunctor2017a,   DrinfeldGeometricconstantterm2016, SchiederGeometricBernsteinAsymptotics2016, Wangreductivemonoidassociated2017}) of the wonderful compactification and Vinberg semigroup in representation theory, many inspired by the seminal papers~\cite{BezrukavnikovGeometrysecondadjointness2015} and~\cite{SakellaridisPeriodsHarmonicAnalysis2017} on asymptotics and scattering in the $p$-adic setting and relating to the earlier~\cite{EmertongeometricJacquetfunctor2004} on Jacquet functors in the real setting. 

The setting of this paper bears a strong similarity to that of~\cite{AbeJacquetFunctorConcini2015,EmertongeometricJacquetfunctor2004, ChenformulageometricJacquet2017, BezrukavnikovCharactermodulesDrinfeld2012}, in which nearby cycles on wonderful compactifications or Vinberg degenerations are combined with Beilinson--Bernstein localization. Specifically, ~\cite{AbeJacquetFunctorConcini2015,EmertongeometricJacquetfunctor2004,ChenformulageometricJacquet2017} use the (usual) Beilinson--Bernstein equivalence on the flag variety to relate Harish-Chandra modules to $B$-equivariant sheaves on $G/K$ (e.g., in the group case $B\times B$-equivariant sheaves on $G$) and then use the degeneration to $B$-equivariant sheaves on $G/MN$ (e.g., the horocycle space) to realize the Jacquet functor. Our construction applies the same degeneration to a completely different class of sheaves on $G/K$ (or $G$) -- the differential equations satisfied by spherical functions (in the group case, matrix coefficients), which only relate to $B$-equivariance in the case of category $\cO$. Thus, the relation to representations and localization is different. In particular, our picture has a more direct relation to the study of matrix coefficients and harmonic analysis. Our results are also more elementary and algebraic, thereby circumventing topology (perverse sheaves and nearby cycles) and dropping requirements of holonomicity, admissibility, fixed infinitesimal character, etc.
 
We also mention recent work of W.~Li \cite{li_regularity_2020}, which introduces the notion of a $K$-admissible $D$-module on a homogeneous $G$-variety, where $K$ is a subgroup of the real group $G$, and relates regular holonomicity of $K$-admissible $D$-modules to spherical properties of $K$ and $Z$. 

Following the ideas of Ginzburg, for a subgroup K of a connected reductive R-group G we introduce the notion of K-admissible D-modules on a homogeneous G-variety Z. We show that K-admissible D-modules are regular holonomic when K and Z are absolutely spherical.

\subsection{Acknowledgments} 

We would like to acknowledge and thank David Nadler, whose ideas heavily influenced our work. This paper grew out of discussions with David Nadler, Jonathan Block, and Nigel Higson as part of the SQuaRE ``The Baum--Connes conjecture and geometric representation theory'' at the American Institute of Mathematics, whose support we are grateful for. In particular we would like to thank Nigel Higson for introducing us to the theory of asymptotics of matrix coefficients and related aspects of harmonic analysis on real groups. We would also like to thank Gwyn Bellamy,  Sam Gunningham, and Quoc P. Ho for numerous enlightening conversations, and an anonymous reviewer for a close reading and detailed feedback. 
DBZ would like to acknowledge the National Science Foundation for its support through individual grants DMS-1103525 and DMS-1705110. IG is grateful for the support of the Advanced Grant ``Arithmetic and Physics of Higgs moduli spaces'' No.\ 320593 of the European Research Council. 

\section{Preliminaries}\label{sec:preliminaries}

For a smooth variety $X$ over $\C$, we denote by  $\Theta_X$ the tangent sheaf, by $\D_X$ the sheaf of differential operators, by $D_X$ the algebra $\Gamma(X, \D_X)$ of global differential operators, and by  $D(X)$ the (abelian) category of $D$-modules on $X$, that is, $\D_X$-modules. Throughout, $G$ denotes a connected semisimple algebraic group over $\C$. Let  $\g = \Lie(G)$ be its Lie algebra, $U\g$ the universal enveloping algebra, and $\fZ$ the center of $U\g$. The coinvariants of a $\g$-module $V$ are defined as the zeroth homology of $\g$ with coefficients in $V$, i.e.,
 $(V)_\fg := V\ot_{U\g} \C\simeq H_0(\g, V).$ 
  
\subsection{Localization and equivariant sheaves}\label{subsec:locequiv} 

Suppose $G$ acts on a smooth variety $X$.  For $x \in X$, write $\bar u_x$ for the image of $u \in \g$ under the differential of the map $a_x : G \rightarrow X;$  $g \mapsto g\inv\cdot x$. The vector field $\bar u$ is called the infinitesimal action of $u \in \g$ on $X$, and we obtain a map of Lie algebras $\g \rightarrow \Gamma(X, \Theta_X)$. The cotangent bundle $T^*X $ has a Hamiltonian symplectic structure with moment map $T^*X \rightarrow \g^*$ given by $(x,\alpha_x) \mapsto [ u \mapsto \langle \alpha_x, \bar u_x\rangle]$. The infinitesimal action extends to an algebra homomorphism
$$\mu: U\g \rightarrow D_X = \Gamma(X, \D_X),$$
which we refer to as a `quantum moment map' or just `moment map' (see e.g.\ \cite{LuMomentmapsquantum1993}). The adjoint action of $u \in U\g$ on $a \in D_X$ is defined as $ u \triangleright a = \mu(u_{(1)}) \cdot a \cdot \mu(S(u_{(2)})).$ The multiplication on $D_X$ is $U\g$-linear, and hence $D_X$ is an algebra  in the tensor category $U\g\dmod$.  

\begin{definition} The localization functor for the action of $G$ on $X$ is defined as 
$$ \Loc_X : U\g\dmod \rightarrow D(X); \qquad M \mapsto \D_X \ot_{U\g} M.$$
\end{definition}

The infinitesimal action defines a a map of Lie algebroids $\fg \ot \O_X \rightarrow \Theta_X$ known as the anchor map. We denote the kernel of the anchor map as  $\underline{\mathfrak s}$. 

\begin{lemma}\label{lem:coinvariants} Suppose $G$ acts  transitively on $X$. Then the localization of  a $U\g$-module $V$ on $X$ is equal to the $\underline{\fraks}$-coinvariants of $V \ot \O_X$: $$\Loc_X(V) = (V \ot \O_X )_{\underline{\fraks}}.$$  Moreover, the fiber of  $\Loc_X(V)$ at $x$ is equal to the coinvariants of $V$ with respect to the Lie algebra of the stabilizer of $x$ in $G$.   \end{lemma}

\begin{proof} From the anchor map, we obtain an exact sequence of sheaves on $X$:
	$$ 0 \rightarrow \langle \underline{\fraks} \rangle  \rightarrow U\fg \ot \O_X\rightarrow \D_X \rightarrow 0,$$
which is exact on the right since $G$ acts on $X$ transitively. Consequently, for any $U\g$-module $V$ we obtain:
	$$ \underline{\fraks}( V \ot \O_X)   \rightarrow V \ot \O_X\rightarrow \Loc_X(V) \rightarrow 0.$$
The fiber of  $\underline{\fraks}$ over $x\in X$ is equal to the Lie algebra of the stabilizer of $x$ in $G$. Since these stabilizers are all conjugate, $\underline{\fraks}$ is flat as a coherent sheaf on $X$. The last claim  follows. \end{proof}

\begin{example}\label{ex:matrixcoef} Consider the action of $G \times G$ on $G$ by left and right multiplication: $(x,y) \triangleright g$ $=$ $xgy\inv$ for $x,y,g \in G$. The localization functor
	$$\Loc_G : U(\g \times\g)\dmod \rightarrow D(G)$$
is called `matrix coefficients localization' (see Section \ref{sec:intro:wonderful} of the intro). The stabilizer of a point $g \in G$ is given by: $\{ (h, g\inv h g) \in G \times G : h \in G\} $, and its Lie algebra is given by: $\fraks_g = \{ (x, \ad_{g\inv}(x) ) \in \g \times \g : x \in \g\}.$ Thus, the fiber of the localization $\Loc_G(W)$ of a $U\g \ot U\g$-module $W$ at $g \in G$ is the $\fraks_g$-coinvariants of $W$, i.e.\ the quotient of $W$ by the subspace generated by elements of the form $(x, \ad_{g\inv}(x))\cdot w$ for $w \in W$ and $x \in \g$. It is clear that the matrix coefficients localization functor factors through the category $U\g \ot_\fZ U\g\dmod$. 
\end{example}

A $G$-equivariant sheaf on $X$ is a quasi-coherent sheaf $\cF$ on $X$ equipped with an isomorphism $I: a^* \cF \rightarrow p^* \cF$, where $a, p : G \times X \rightarrow X$ are the action and projection maps, respectively, and $I$  satisfies a certain associativity constraint. We refer the reader to \cite{ChrissRepresentationTheoryComplex2009} and \cite{TakeuchiDModulesPerverseSheaves2007} for background information on equivariant sheaves. The structure sheaf $\O_X$ and the sheaf of differential operators $\D_X$ have natural $G$-equivariant structures. 

\begin{definition} A weakly $G$-equivariant $D$-module is a $\D_X$-module $\cF$ on $X$ equipped with a $G$-equivariant structure   $I: a^* \cF \rightarrow p^* \cF$ that is an isomorphism of $\O_G \boxtimes \D_X$-modules\footnote{If, furthermore,  $I$ is an isomorphism of $\D_G \boxtimes \D_X$-modules, then $\cF$ is called a strongly $G$-equivariant $D$-module. We will not use strongly equivariant $D$-modules in this paper.}. We denote the category of such by $D_G(X)$.   \end{definition}

Thus, a weakly $G$-equivariant $D$-module is a $\D_X$-module with a compatible $G$-equivariant structure. If $f : X \rightarrow Y$ is an equivariant map between smooth $G$-varieties, then the pullback functor $f^*$ of $D$-modules factors through the weakly $G$-equivariant categories: $f^* : D_G(Y) \to D_G(X).$

Suppose that $K_1$ and $K_2$ are algebraic groups, and that the product $K_1 \times K_2$ acts on a smooth variety $X$. We have an algebra homomorphism $\mu : U(\frakk_1) \ot U(\frakk_2) \rightarrow D_X.$ The adjoint action of $u \in \U(\frakk_1)$ on $a \in D_X$ is given by $ u \triangleright a$ $=$ $\mu(u_{(1)} \ot 1) \cdot a \cdot \mu(S(u_{(2)}) \ot 1).$ The right multiplication action of $v \in \U(\frakk_2)$ on $a \in D_X$ is given by $  a \triangleleft v = a \cdot \mu(1 \ot v).$

\begin{lemma}\label{lem:commuting} The localization functor for the action of $K_2$ factors through the category  of weakly $K_1$-equivariant $D$-modules on $X$, inducing a functor $\U(\frakk_2)\dmod \rightarrow D_{K_1}(X)$.\end{lemma}

\begin{proof}  We define a $\U(\frakk_1)$-action on $\Loc(M) = \D_X \ot_{\U(\frakk_2)} M$ via the adjoint action of $\U(\frakk_1)$ on $\D_X$. This action well-defined since the two actions defined in the preceding definition commute, and it is locally finite. \end{proof}

\begin{example} If $K_1 = (\C^\times)^r$ is a torus, then weakly equivariant $D$-modules on $X$ are also known as monodromic, conical, or homogeneous $D$-modules, and alternative notation is $D\submon(X)$. We obtain a $\Z^r$-grading on the algebra  $D_X$. One checks that the quantum moment map for $U(\frakk_2)$ lands in the zero-th graded piece of $D_X$.  \end{example}

\begin{example} Continuing Example \ref{ex:matrixcoef} from above, we have that $K_2 = G \times G$ acts by left and right multiplication on $G$, and this action commutes with the multiplication action of the center $K_1 = Z(G)$. Thus, we obtain a localization functor $\Loc_G : U\g \ot U\g\dmod \to D_{Z(G)}(G)$.  \end{example}

\subsection{Relative differential operators}\label{subsec:rel-diff-ops}

Let $A$ be a commutative ring. For an $A$-bimodule $M$, let $\Derv(A, M)$ denote the $A$-module of derivations from $A$ to $M$. When $M = A$, we abbreviate $\Derv(A,A)$ by $\Derv(A)$. If $A \rightarrow B$ is a map of commutative rings, then, for a $B$-bimodule $N$, let $\Derv_A(B, N)$ denote the $B$-module of relative derivations from $B$ to $N$, that is, derivations that vanish on elements in the image of $A$ in $B$. A sequence of maps $A \rightarrow B \rightarrow C$ between commutative rings induces, for any $C$-bimodule $L$, an exact sequence of $C$-modules: \begin{equation}\label{eq:dervses} 0 \rightarrow \Derv_B(C,L) \rightarrow \Derv_A(C,L) \rightarrow \Derv_A(B,L).\end{equation}
Relative derivations are closely related to relative differential operators, as we now explain. 

Let $X = \Spec(S)$ and $Y = \Spec(R)$ be smooth affine varieties over $\C$  and $f : X \rightarrow Y$ be a flat map. We have a corresponding map of commutative  $\C$-algebras  $R \to S$. 

\begin{definition}\label{def:rel-diff-ops}  The sheaf $\D_f$ of relative differential operators on $X$ (with respect to $f$) is defined as the quasi-coherent sheaf on $X$ corresponding to the subalgebra of $\End_\C(R)$ generated by $\Derv_R(S, S)$ and the left multiplication operators $m_a : b \mapsto ab$ for $a \in S$.     \end{definition}

The sheaf $\D_f$ coincides with the subsheaf of $\D_X$ generated by vector fields on $X$ that commute with functions pulled back from $Y$. In the special case that  $Y = \Spec(\C)$ is a point, we have that  $\D_f$ is full sheaf $\D_X$ of differential operators on $X$.  The proof of the following lemma follows from definitions.	
	
\begin{lemma}\label{lem:reldiff-qmm} Suppose a linear algebraic group $G$  acts on $X$ and on $Y$, and  $f$ is $G$-equivariant.  If $G$ acts trivially on $Y$, then the quantum moment map $U\g \rightarrow \Gamma(X, \D_X)$ factors through the subalgebra $\Gamma(X, \D_f)$. 
\end{lemma}	

For the remainder of this section, fix an open subvariety $U$ of $X = \Spec(S)$ such that the restriction  $ {f}^{\circ} := f|_U : U \to Y$ of $f$ to $U$ is smooth. In particular, for each $y\in Y$, the fiber $U_y$ of $f^\circ$ over $y$ is smooth. 

\begin{definition}\label{def:rel-diff-ops-U} The sheaf $\D_{f^\circ}$ of relative differential operators on $U$ with respect to $f^\circ$ (and the embedding $U \hookrightarrow X$) is defined as the restriction of $\D_f$ to $U$:
	$$\D_{f^\circ} = \D_f |_U.$$	
\end{definition}


\begin{lemma}\label{lem:rel-diff-ops-fiber}
	Let $y \in Y$. The restriction of the sheaf of relative differential operators $\D_{f^\circ}$ to  the fiber $U_y $ coincides with the sheaf of differential operators on ${U_y}$.
\end{lemma}

\begin{proof}[Sketch of proof.] Recall that, by the choice of $U$, the fiber $U_y$ is smooth, so the sheaf of differential operators thereon is well-defined. To prove the claim, one  reduces to the case where $X = U$ and $f$ is smooth. In this case, the Kähler differentials $\Omega_{S/R}$ define a projective $S$-module.  The result follows from the following general facts: First, base change for Kähler differentials gives an isomorphism $\Omega_{S/R} \ot_R \C_y \stackrel{\sim}{\longrightarrow} \Omega_{S_y}$, where $\C_y$ is the one-dimensional $R$-module corresponding to the closed point $y \in Y$, and $S_y = S \ot_R \C_y$.   Second, we have natural isomorphisms $\Derv_R(S,S) = \Hom_R(\Omega_{S/R}, S)$ and $\Derv_\C(S_y, S_y) = \Hom_\C(\Omega_{S_y}, S_y)$.
\end{proof}

 Now suppose a linear algebraic group $G$  acts on $X$ and on $U$, and  that  the embedding $U \hookrightarrow X$ is $G$-equivariant. Suppose further that $G$ acts on $Y$, and the map $f : X \to Y$ is $G$-equivariant. Hence $f^\circ$ is also $G$-equivariant, and there is an action of $G$ on each fiber $U_y$. 

\begin{lemma} Suppose the action of $G$ on $Y$ is trivial.  For any $y \in Y$,  localization on $U$ commutes with restriction to the fiber $U_y$; more precisely, the following diagram commutes:
		\[\xymatrix{U\g \ar[r] \ar[rd] & \Gamma(X,\D_{f^\circ}) \ar[d] \\ & \Gamma(U_y, \D_{U_y})  }  \]
		where the maps emanating from $U\g$ are the quantum moment maps, and the vertical map is restriction. 
\end{lemma}
 
 \begin{proof} Let $x\in U_y$. Since the fiber $U_y$ is a smooth $G$-stable subvariety of $X$, the natural map $\g \rightarrow T_x X$ factors through the subspace $T_x U_y$. In other words, the infinitesimal action of $\g$ at $x$ consists of vector fields on $U_y$. Any function on $X$ that is pulled back from $Y$ is constant on the fibers of $f$, hence its restriction to $U_y \subseteq f\inv(y)$ is annihilated by any vector field on $U_y$. In particular, it is annihilated by the vector fields forming the infinitesimal action of $\g$. These vector fields generate the image of the quantum moment map. \end{proof}

\section{Relative differential operators and Rees spaces}\label{sec:relativediff}

\subsection{Algebraic set-up}\label{subsec:algfiltrations}

Fix a positive integer $r$ and let $\Phi$  be a lattice of rank $r$. Fix a linearly independent set $\{\beta_1, \dots, \beta_r\}$ of elements of $\Phi$. Consider the  full-rank sublattice  $\Z\langle \beta_1, \dots, \beta_r \rangle$ of $\Phi$ generated by the $\beta_i$. The set $\Gamma$ of cosets of this sublattice in $\Phi$ forms a finite group.  Define a partial order on $\Phi$ by setting $\mu \leq \lambda$ if and only if $\lambda - \mu$ is a linear combination of the $\beta_i$ with non-negative coefficients; in symbols:
\begin{equation}\label{eqn:partialorder} \mu \leq \lambda \quad \Leftrightarrow \quad \lambda - \mu = \sum_{i=1}^r  n_i \beta_i \quad \text{with $n_i \geq 0$ for all $i$}. \end{equation}

\begin{lemma}\label{lem:parialorder}  If $\phi_1$ and $\phi_2$ belong to the coset $\gamma \in \Gamma$, then there exists $\phi \in \gamma$ such that $\phi_1 \leq \phi$ and $\phi_2 \leq \phi$. If $\gamma_1$ and $\gamma_2$ are distinct cosets in $\Gamma$, then no element of $\gamma_1$ is related to an element of $\gamma_2$ under the partial order $\leq$. 
 \end{lemma}

\begin{proof} To prove the first statement, suppose $\phi_1$ and $\phi_2$ belong to the coset $\gamma \in \Gamma$. Then $\phi_2 = \phi_1 + \sum_i n_i \beta_i$ for some $n_i \in \Z$, and  $\phi = \phi_1 + \sum_i \max(n_i, 0) \beta_i$ satisfies the desired property. For the second statement, the contrapositive is clear from \ref{eqn:partialorder}.\end{proof}

\begin{definition}\label{def:filteredvs}
	 A $(\Phi, \leq)$-filtration on a vector space $V$ is the data of a subspace $V_{\leq \lambda} \subseteq V$ for every $\lambda \in \Phi$ subject to the following conditions:
	\begin{itemize}
		\item  If $\mu \leq \lambda$, then $V_{\leq \mu }   \subseteq V_{\leq \lambda}$. 
		\item  $V$ admits a direct sum decomposition $V = \bigoplus_{\gamma \in \Gamma} \left( \bigcup_{\lambda \in \gamma } V_{\leq \lambda} \right)$.
	\end{itemize}
\end{definition}

When the partial order is clear from context, we write `$\Phi$-filtration' instead of `$(\Phi, \leq)$-filtration'. Let $\C[\Phi]$ denote the group algebra of $\Phi$, which is generated by elements $t^\lambda$ for $\lambda \in \Phi$ subject to the condition $t^{\lambda}  \cdot t^{\mu} = t^{\lambda + \mu}$. For $\gamma \in \Gamma$, let $\C[\gamma]$ be the vector subspace of $\C[\Phi]$ spanned by all $t^\lambda$ with $\lambda \in \gamma$. Then $\C[\Phi] = \bigoplus_{\gamma \in \Gamma} \C[\gamma]$. 

\begin{definition}\label{def:filteredalgebra}	Let $A$ be a unital, associative algebra over $\C$. A $\Phi$-filtration on $A$ is the data of a $\Phi$-filtration on $A$ satisfying the following condition: if $a \in A_{\leq \mu}$ and $b \in A_{\leq \lambda}$, then the product $ab$ belongs to $A_{\leq \mu + \lambda}$.  The  Rees algebra corresponding to $A$ is defined as the following subalgebra of  $A \ot \C[\Phi]$: $$\Rees(A) = \bigoplus_{\lambda \in \Phi} A_{\leq \lambda} t^\lambda.$$
The associated graded algebra $\gr(A)$ of a $\Phi$-filtered algebra is the $\Phi$-graded algebra whose $\lambda$-th graded piece is the quotient of $A_{\leq \lambda}$ by the the union of all subspaces $A_{\leq \mu}$ with $\mu \leq \lambda$. 
\end{definition}

We note that, for each $\gamma\in \Gamma$, the subspace $A_\gamma = \bigcup_{\lambda \in \gamma} A_{\leq \lambda}$ is a subalgebra of $A$, and is non-unital  unless $\gamma$ is the identity element of $\Gamma$. Moreover,  $\Rees(A) = \bigoplus_{\gamma \in \Gamma} \Rees(A_\gamma)$.  The algebra $A$  is unital with $1 \in A_{\leq 0}$, and hence $1 \in A_{\leq \beta_i}$ for any $i$. Thus, there is an algebra homomorphism
\begin{equation}\label{eqn:maptoRees}
\C[t^{\beta_1}, t^{\beta_2}, \dots, t^{\beta_r}] \longrightarrow \Rees(A) 
\end{equation}
from the subalgebra $\C[t^{\beta_1}, t^{\beta_2}, \dots, t^{\beta_r}]$ of  $\C[\Phi]$ generated by the $t^{\beta_i}$ (this is a polynomial algebra in $r$ variables) to $\Rees(A)$, taking $t^{\beta_i} $ to $1 \cdot t^{\beta_i}  \in A_{\leq \beta_i} t^{\beta_i}$.  For $p = (p_1, p_2, \dots, p_r) \in \Z^r$, let  $A_p$ be the quotient of $\Rees(A)$ by the ideal generated by the elements $t^{\beta_i}-p_i$ for $i = 1, \dots, r$. The following lemma follows by direct computation. 

\begin{lemma}
If $p_i \neq 0$ for all $i$, then  $A_p$ is isomorphic to $A$.  If $p=0$, then $A_0 = \gr (A)$. The localization $\Rees(A)[(t^{\beta_1})\inv , (t^{\beta_2})\inv, \dots,  (t^{\beta_r})\inv]$ is isomorphic to $\bigoplus_{\gamma \in \Gamma} A_\gamma \ot \C[\gamma]$. 
\end{lemma}

\begin{definition} 	Let $A$ be an algebra over $\C$ with a $\Phi$-filtration. A filtered left $A$-module is a left $A$-module $M$ together with the data of $\Phi$-filtration on $M$ satisfying the following condition:  if $a \in A_{\leq \mu}$ and $m \in M_{\leq \lambda}$, then the element $am $ belongs to $M_{\leq \mu + \lambda}$. One defines the notions of a filtered right $A$-module and a filtered $A$-bimodule similarly. 
\end{definition}

Let $A\dmod\filt$ denote the category of filtered (left) $ A$-modules. We also have the categories $\Rees(A)^\gr$ and $\gr (A)\dmod^\gr$ of graded modules for $\Rees(A)$ and the associated graded $\gr(A)$ of $A$. There is a `Rees' functor
$$\rho : A\dmod\filt \longrightarrow \Rees(A)\dmod^\gr$$ taking a filtered $A$-module $M$ to the graded $\Rees(A)$-module $\bigoplus_\lambda M_{\leq \lambda} t^\lambda$. There is also a `restriction to zero' functor 
$$\zeta :  \Rees(A)\dmod^\gr  \longrightarrow \gr(A)\dmod^\gr $$ 
obtained by evaluating each $t^{\beta_i}$ to zero. The composition $\gr = \zeta \circ \rho$ takes a filtered $A$-module to its associated graded as a $\gr(A)$-module, and we have the following commutative diagram:

\begin{equation} \label{diag:reesgr} \xymatrix{ & \Rees(A)\dmod^\gr \ar[ld]_{  \zeta  }  &  \\ \gr(A)\dmod^\gr &   & A\dmod\filt  \ar[ul]_{\rho} \ar[ll]^{\gr} }  \end{equation}

\subsection{Derivations and Rees algebras}\label{subsec:derviRees}

Let $A$ be an algebra with a $\Phi$-filtration. For $\lambda \in \Phi$, consider the following subspace
$$\Derv(A)_{\leq \lambda} = \{ \theta \in \Derv(A) \ | \ \theta(A_{\leq \lambda^\prime}) \subseteq A_{\leq \lambda^\prime + \lambda} \ \text{\rm for all $\lambda^\prime$ in $\Phi$} \}$$ 
of the module $\Derv(A) = \Derv_\C(A,A)$ of derivations of $A$. Observe that the subspace $\bigoplus_{\gamma \in \Gamma}  \left( \bigcup_{\lambda \in \gamma} \Derv(A)_{\leq \lambda } \right) $ of $\Derv(A)$ is a filtered $A$-module. Although this  submodule is proper in general,  in all our intended applications it will turn out to be all of  $\Derv(A)$ (c.f.\ Corollary \ref{cor:derivfilt}). We will consider the $A \ot \C[\Phi]$-submodule of $\Derv(A) \ot \C[\Phi]$ given by $\bigoplus_{\lambda \in \Phi} \Derv(A)_{\leq \lambda}t^\lambda$. 

We abbreviate the polynomial algebra $\C[t^{\beta_1}, t^{\beta_2}, \dots, t^{\beta_r}] $   by $\C[\{t^{\beta_i}\}]$. The homomorphism $\C[\{t^{\beta_i}\}] \rightarrow \Rees(A)$  from \ref{eqn:maptoRees} induces a short exact sequence of $\Rees(A)$-modules:
$$ 0 \rightarrow \Derv_{\C[\{t^{\beta_i}\}]}(\Rees(A)) \rightarrow \Derv(\Rees(A)) \rightarrow \Derv\left(\C[\{t^{\beta_i}\}], \Rees(A)\right).$$   A derivation $\theta$ of $\Rees(A)$ belongs to the relative differentials $\Derv_{\C[\{t^{\beta_i}\}]}(\Rees(A))$ if and only if $\theta(t^{\beta_i}) =0$ for  $i = 1, \dots, r$. 

\begin{prop}\label{prop:rel-Rees=Rees-rel} There is an isomorphism of $\Rees(A)$-modules
	$$ \tau:  \bigoplus_{\lambda\in \Phi} \Derv(A)_{\leq \lambda}t^\lambda  \stackrel{\sim}{\longrightarrow} \Derv_{\C[\{t^{\beta_i}\}]} (\Rees(A))$$
	where, for  $\theta \in \Derv(A)_{\leq \lambda}$, the derivation $\tau(\theta t^\lambda)$  takes $a t^\mu \in  A_{\leq \mu}t^\mu$  to $\theta(a)t^{\mu + \lambda}$. \end{prop}

\begin{proof} We give the proof in the case $r = 1$. The general argument follows similarly.  First we show that $\tau$ is well defined. Let  $\theta \in \Derv(A)_{\leq n}$, $a \in A_{\leq m}$, and $b \in A_{\leq k}$. Then 
	\begin{align*} \tau(\theta t^n)(at^m b t^k)  = \theta(ab) t^{n+m+k} &= \theta(a) b  t^{n+m+k} + a \theta(b)  t^{n+m+k} \\  &= \left( \tau(\theta t^n)(at^m) \right) b t^k + at^m \left( \tau(\theta t^n)(b t^k)\right),\end{align*}
	so $\tau(\theta t^n)$ is a derivation. Moreover, $\tau(\theta t^n)(t) =\theta(1)t^{n+1} = 0$ since $\theta$ is a derivation of $A$. Thus  $\tau(\theta t^n)$  is a relative derivation. 
	
	To see that $\tau$ is injective, let  $\sum_{n \in \Z} \theta_n t^n$ be an element of $\bigoplus_{n \in \Z} \Derv(A)_{\leq n}t^n$ and suppose  $\tau\left(\sum_{n \in \Z} \theta_n t^n\right) =0$ . Then $\sum_{n \in \Z} \theta_n(a) t^{n +m} = 0 $ for all 	 $a  \in  A_{\leq m}$. Hence $\theta_n(a) = 0 $ for all $n \in \Z$, and thus $\theta(a) =0$ for all $a \in A_{\leq m}$. It follows that $\theta=0$. 
	
	Before proving surjectivity, we set  $\pi_k : \Rees(A) \rightarrow A_{\leq k}$ to be the projection onto the $k$-th graded piece. The following diagram commutes:
	\begin{equation}\label{diag:projmultt}  \xymatrix{\Rees(A) \ar[rr]^{ - \cdot t^{N\beta_1}  } \ar[d]_{\pi_k} & &  \Rees(A)  \ar[d]^{\pi_{k +N\beta_1}}\\ A_{\leq k} \ar@{^{(}->}[rr] & & A_{\leq k + N\beta_1}  }  \end{equation}
	where the top horizontal map is multiplication by $t^{N\beta_1}$ and the bottom horizontal map is the inclusion. 
	
	We now show that $\tau$ is surjective. Let   $D \in \Derv_{\C[t^{\beta_1}]}(\Rees(A))$. For $n, m \in \Z$, define a linear map $\theta_{n,m} : A_{\leq m} \rightarrow A$ by:
	$$a \mapsto (\pi_{n+m} \circ D) (at^m).$$
	In other words, $\theta_{n,m}$ assigns to $a \in A_{\leq m}$,the coefficient of $t^{n+m}$ in $D(a t^m)$.  We argue that the maps $\theta_{n,m}$ are compatible on the filtered pieces of $A$, and assemble to a well-defined map  $\theta_n : A \to A$, which is in fact a derivation, and  belongs to $\Derv(A)_{\leq n}$. 

	\begin{itemize}
		\item 	To prove that the $\theta_{n,m}$ are compatible, we demonstrate that, for any $m \in \Z$ and any positive integer $N$, the restriction of $\theta_{n, m + N \beta_1}$ to $A_{\leq m} \subseteq A_{\leq m + N\beta_1}$ is equal to $\theta_{n,m}$.  Indeed, let  $a \in A_{\leq m} \subseteq A_{\leq m + N\beta_1}$. Then 	
		\begin{align*}
		\theta_{n,m + N\beta_1}(a) &=		\left(\pi_{n+m+{N\beta_1}} \circ D\right) \left(at^{m+{N\beta_1}}\right) =  \pi_{n+m+{N\beta_1}} \left(D \left(at^m t^{N\beta_1} \right)\right)\\ 
		&= \pi_{n+m+{N\beta_1}} \left(D \left(at^m\right) t^{N\beta_1} \right) = \left(\pi_{n+m} \circ D\right)\left(at^m\right) = \theta_{n,m}(a).
		\end{align*}
		where we use the fact that $D$ is a relative derivation in the third equality, and the commutativity of diagram \ref{diag:projmultt} in the fourth equality.
		
		\item To see that $\theta_n$ is a derivation, let $a \in A_{\leq m}$ and $b \in A_{\leq k}$. Then 
		\begin{align*} \theta_n(ab) &=  (\pi_{n+m+k} \circ D) (at^m bt^k) t^{-(n+m+k)} =  \pi_{n+m+k} (  D (at^m) bt^k + at^m D(bt^k)) t^{-(n+m+k)} \\ &=  (\pi_{n+m+k} \circ  D) (at^m) b  t^{-(n+m)}  + a (\pi_{n+k} \circ D)(bt^k) t^{-(n+k)}   = \theta_n(a) b + a \theta_n(b). \end{align*}
		
		\item To see that $\theta_n$ belongs to $\Derv(A)_{\leq n}$, note that if $ a\in A_{\leq m}$, then the definition of $\theta_{n,m}$ implies that $\theta_n(a)$ belongs to $A_{\leq n+m}$. 
		
	\end{itemize}
	Therefore, we have that  $\sum_n \theta_n t^n$ belongs to $\bigoplus_{n \in \Z} \Derv(A)_{\leq n}t^n$. Moreover,  $\tau \left( \sum_n \theta_n t^n  \right) = D$, since, for $a \in A_{\leq m}$ we have that $$\tau\left( \sum_n \theta_n t^n\right)(a t^m) = \sum_n \theta_n(a) t^{n+m} = \sum_n \left(\pi_{n+m} D\right) (a) = D(a).$$ We have shown that $\tau$ is surjective. This completes the proof.   \end{proof}

\subsection{Localization commutes with associated graded}\label{sec:loccommute}

Let $A$ be a commutative $\Phi$-filtered algebra over $\C$ and let $X = \Spec(A)$ be the corresponding affine scheme. As noted in Section \ref{subsec:algfiltrations} above, the subalgebra  of  $\C[\Phi]$ generated by the $t^{\beta_i}$ is a polynomial algebra in $r$ variables, and hence we identify it in a natural way with the coordinate algebra $\O(\A^r)$ of affine $r$-space. The algebra homomorphism \ref{eqn:maptoRees} induces a map of affine schemes
$$\VinbX :=  \Spec(\Rees(A)) \rightarrow \A^r$$
whose generic fiber is $X = \Spec(A)$ and whose fiber over $0$ is $\Spec(\gr (A))$.  

Recall that $\Phi$ denotes a lattice of rank $r$. We identify its group algebra $\C[\Phi]$ with the coordinate algebra of an algebraic torus $T$ of rank  $r$. The spectrum of any commutative  algebra graded by $\Phi$ has an action of the torus $T$. In particular, the affine schemes $\VinbX$ and $\Spec(\gr (A))$ both carry $T$-actions.  Since we would like to consider $D$-modules on these spaces, as well as differential operators relative to the map $\VinbX  \rightarrow \A^r$, we henceforth make the following assumptions:
\begin{enumerate}
	\item  $X = \Spec(A)$ is smooth. 
	
	\item $\VXcirc$ is a $T$-equivariant open subvariety of $\VinbX= \Spec(\Rees(A)) $ such that the restriction $ \pi : \VXcirc \rightarrow \A^r$ satisfies:
	\begin{itemize}
		\item the map $\pi$ smooth. 
		\item  the generic fiber of $\pi$ is $X$.
	\end{itemize}
\item We have:	$$\Derv(A) = \bigoplus_{\gamma \in \Gamma}  \left( \bigcup_{\lambda \in \gamma} \Derv(A)_{\leq \lambda } \right),$$
where, as in Section \ref{subsec:algfiltrations}, $\Gamma$ is the group of cosets of the sublattice $\Z\langle \beta_1, \dots, \beta_r \rangle$ of $\Phi$. 
\end{enumerate}

Recall that the algebra $D_X = \Gamma(X, \D_X)$ of global differential operators on $X$ is the subalgebra of $\End_\C(A)$ generated by $A$ (acting by left multiplication) and $\Derv(A)$.  Assumption (3) above implies that there is a $\Phi$-filtration on $\Derv(A)$ compatible with that on $A$. Hence $D_X$ also carries a $\Phi$-filtration, and we can consider its Rees algebra $\Rees(D_X)$. Let $\Delta : \Rees(A)\dmod \longrightarrow \QCoh(\VinbX)$ denote the localization functor, which is an equivalence (with inverse given by the global sections functor).  Let $\mathcal R \D_X = \Delta(\Rees(D_X))$ denote the sheaf of rings on $\VinbX$  determined by $\Rees(D_X)$. 

Observe that we are in the setting of Section \ref{subsec:rel-diff-ops}, where ${\mathbb{V}}^\circ_X$, ${\mathbb{V}}_X$, and $\A^r$ play the roles of $U$, $X$, and $Y$, respectively. Consequently, we have the sheaf $\D_\pi$ of differential operators on $\VXcirc$ relative to $\pi$ (Definitions \ref{def:rel-diff-ops} and \ref{def:rel-diff-ops-U}).
	
\begin{lemma}\label{lem:resRDX=Dpi}
	The restriction  of the sheaf  $\mathcal R \D_X$ to $\VXcirc$ is coincides with the sheaf $\D_\pi$  of  differential operators on $\VXcirc$ relative to the map $\pi$. 
\end{lemma}

\begin{proof}
	By definition, $\D_\pi$ is the restriction to $\VXcirc$ of the subalgebra of $ \Delta(\End(\Rees(A)))$ generated by $\Delta(\Rees(A))$ acting by left multiplication and $\Delta(\Derv_{\C[\{t^{\beta_i}\}]} (\Rees(A)  ))  = \Delta(\Rees(\Derv(A)  ))$, where we use Proposition  \ref{prop:rel-Rees=Rees-rel} in the last equality. Hence it is the restriction to $\VXcirc$ of the image under $\Delta$ of the subalgebra of $\End_\C(\Rees(A))$ generated by $\Rees(A)$ and $\Rees(\Derv(A)  )$. On the other hand, $\mathcal R \D_X$ is the image under $\Delta$ of $\Rees(D_X)$, which is exactly the subalgebra of  $\End_\C(\Rees(A))$ generated by $\Rees(A)$ and $\Rees(\Derv(A)  )$. \end{proof}

Let $X_0$ be the fiber of $\pi$ over $0$. The sheaves $\mathcal R \D_X$,  $\D_{\pi}$,  and  $\D_{X_0}$ are naturally $T$-equivariant, and we have the corresponding categories of graded (i.e.\ weakly $T$-equivariant) modules:  $\mathcal R\D_X\dmod^\gr$, $D(\pi)^\gr$, and $D(X_0)^\gr$, respectively.  Let $D(X)\filt$ be the category whose objects are $\D_X$-modules $\cM$ equipped with a $\Phi$-filtration on the global sections   $\Gamma(X, \cM)$ making the latter a $\Phi$-filtered $\D_X$-module.   We note that, since $X = \Spec(A)$ is affine, the category $D(X)\filt$ is equivalent to the category $D_X\dmod\filt$.  There is a `Rees' functor $D(X)\filt \longrightarrow \mathcal R \D_X\dmod^\gr$ taking a filtered $D$-module $\cM$ to the graded $\mathcal R \D_X$-module $\bigoplus_\lambda \cM_{\leq \lambda} t^\lambda$.  Composing with the restriction to ${\mathbb{V}}^\circ_X$, we obtain (via Lemma \ref{lem:resRDX=Dpi}) a functor:
$$ \rho : D(X)\filt \longrightarrow D(\pi)^\gr$$
Let $\iota : X_0 \hookrightarrow \VXcirc$ be the inclusion (which is $T$-equivariant). By Lemma \ref{lem:rel-diff-ops-fiber}, the restriction $\iota^* \D_\pi$ of $\D_\pi$ to $X_0$ is equal to the sheaf $\D_{X_0}$ of differential operators on $X_0$.  Thus we obtain a functor $$\zeta : D(\pi)^\gr \longrightarrow D(X_0)^\gr$$
given by pulling back along $\iota$. The composition $\gr = \zeta \circ \rho : D(X)\filt \longrightarrow D(X_0)^\gr$ computes the associated graded of a filtered $D$-module on $X$, and we have the following commutative diagram:
\begin{equation} \label{diag:reesgrD} \xymatrix{ & D(\pi)^\gr \ar[ld]_{  \zeta  }  &  \\ D(X_0)^\gr &   & D(X)\filt  \ar[ul]_{\rho} \ar[ll]_{\gr} }  \end{equation}

Now we suppose that a linear algebraic group $K$ acts on $X= \Spec(A)$ such that the corresponding action on $A$ preserves the filtered pieces $A_{\leq \lambda}$. Then $K$ acts naturally on ${\mathbb{V}}_X = \Spec(\Rees(A))$, and the $K$-action commutes with the $T$-action. We assume that the subvariety $\VXcirc \subseteq {\mathbb{V}}_X$ is stable under the $K$-action; consequently, we additionally obtain an action of $K$ on $X_0$ commuting with that of $T$. 

\begin{prop}\label{prop:loccommutesassocgr} The localization functors for $X$, $X_0$, and ${\mathbb{V}}^\circ_X$ naturally factor as follows:
$$\U\k\dmod \rightarrow D({X})\filt \qquad \U\k\dmod \rightarrow D({X_0})^\gr \qquad \U\k\dmod \rightarrow D(\pi)^\gr. $$ Moreover, localization commutes with taking associated graded, and we have the following commutative diagram:
\[\xymatrix{ &  & U\k\dmod \ar[d]^{\text{\rm Loc}_{\pi}} \ar[lldd]_{\text{\rm Loc}_{X_0}} \ar[rrdd]^{\text{\rm Loc}_{X}} &  & \\ & & D({\pi})^\gr  \ar[lld]^{ \zeta } &   & \\  D(X_0)^\gr & & & &  D({X})\filt  \ar[llll]^{\gr }  \ar[llu]^{\rho}  }   \]
\end{prop}

\begin{proof} Since the action of $K$ and $T$ on $X_0$ commute, we have from Lemma \ref{lem:commuting} that the localization functor for $X_0$ factors through $D({X_0})^\gr$. Similar considerations, together with Lemma \ref{lem:reldiff-qmm} imply that the localization functor for $\VXcirc$ factors through $D(\pi)^\gr$. 
	
Next, we consider the quantum moment map  $U\k \rightarrow D_X$.  Since each $A_{\leq \lambda}$ is $K$-stable, the map  factors through the zero-th filtered piece, i.e., it induces a map $U\k \rightarrow (D_X)_{\leq 0}$. In particular,  $(D_X)_{\leq \lambda}$ carries an action of $U\k$ for any $\lambda \in \Phi$, and we can form the tensor product $(D_X)_{\leq \lambda} \ot_{U\k} M$ and define a filtration on $D_X \ot_{U \g} M$ as:
$$ (D_X \ot_{U\k} M)_{\leq \lambda} :=  (D_X)_{\leq \lambda} \ot_{U\k} M.$$
Since a $\Phi$-filtration on a $D$-module on $X$ is the same as a $\Phi$-filtration on its global sections, we obtain in this way a $\Phi$-filtration on $\Loc_X(M)$, and a factorization of the localization functor for $X$ through $D({X})\filt$. 

It remains to show that the diagram commutes. We have that $\gr = \zeta \circ \rho$ by definition. We need to verify that $\zeta$ and $\rho$ each commute with the appropriate localization functors. It is straightforward to check the claim for $\zeta$. We give the computation for  $\rho$. Recall that  $\Delta : \Rees(A)\dmod \longrightarrow \QCoh(\mathbb{V}_X)$ denotes the localization functor, which is an equivalence. For any $U\k$-module $M$, we compute:
\begin{align*}
\rho \left( \Loc_{X} (M) \right) & =   \rho\left(\D_X \ot_{U \k} M      \right) =  \Delta \left( \Rees \left(  D_X \ot_{U \k} M   \right)   \right) \biggr\rvert_{{\mathbb{V}}^\circ_X}  =   \Delta \left( \bigoplus_{\lambda \in \Phi} \left(  D_X \ot_{U \k} M   \right)_{\leq \lambda} t^\lambda   \right) \biggr\rvert_{\VXcirc}   \\
& = \Delta \left( \bigoplus_{\lambda \in \Phi}  \left( \left(  D_X   \right)_{\leq \lambda} t^\lambda  \right) \ot_{U \k} M \right)  \biggr\rvert_{{\mathbb{V}}^\circ_X}  = \left( \Delta \left( \bigoplus_{\lambda \in \Phi}  \left(  D_X   \right)_{\leq \lambda} t^\lambda   \right)   \ot_{U \k} M \right)  \biggr\rvert_{\VXcirc} \\
& = \left( \mathcal R \D_X  \ot_{U \k} M \right)  \biggr\rvert_{{\mathbb{V}}^\circ_X} = \left(\mathcal R \D_X  \biggr\rvert_{\VXcirc}\right)   \ot_{U \k} M = \D_\pi \ot_{U \k} M  = \Loc_{\pi} (M)
\end{align*}
\end{proof}

\section{The wonderful compactification}\label{sec:wonderful}

\subsection{Notation}\label{sec:notation}

Let $G$ be a connected semisimple algebraic group over $\C$ with Lie algebra $\g$. Fix a Borel subgroup $B$ and a maximal torus $H$ contained in $B$.  Write $\fb$ and $\fh$ for the corresponding Lie subalgebras of $\g$. The Borel subgroup $B$ has unipotent radical $N:=R_u(B)$, with Lie algebra $\mathfrak{n}$, and it has an opposite Borel subgroup $B^-$ uniquely characterized by the property that $B \cap B^- = H$. Let $N^-$ denote the unipotent radical of $B^-$.  Let $r$ be the rank of $G$. Write $Z = Z(G)$ for the center of $G$, and $\Gad = G/Z(G)$ for the adjoint group of $G$. 

The weight lattice $\Lambda_W$ of $\g$ is generated by the fundamental weights $\omega_1, \dots, \omega_r$. The weight lattice contains the cone $\Lambda_W^+$ of dominant weights. The interior of $\Lambda_W^+$ is the set of regular dominant weights. Thus, dominant weights comprise the nonnegative linear combinations of the fundamental weights, and regular dominant weights comprise the positive linear combinations of fundamental weights. For a root $\alpha \in \fh^*$, we write $\g_\alpha \subseteq \g$ for the corresponding root subspace of $\g$.  Fix a set of positive simple roots $\{ \alpha_1, \dots, \alpha_r\}$ of $H$ relative to $B$. These are linearly independent in $\Lambda_W$ and  generate the root lattice $\Lambda_R$, and we use the set $\Delta = \{1, \dots, r\}$ to index the positive simple roots.

\begin{definition}\label{def:gorder} Define a partial order on $\Lambda_W$ by setting $\mu\leq \lambda$ whenever $\lambda-\mu$ is a  linear combination of positive simple roots with nonnegative coefficients. That is, 
$$ \mu \leq \lambda \quad \Leftrightarrow \quad \lambda - \mu  = \sum_{i=1}^r n_i \alpha_i \quad \text{with $n_i \geq 0$ for all $i$. }$$ 	
Similarly, we write $\lambda < \mu$ if $\lambda \leq \mu$ and $\lambda \neq \mu$. This partial order is referred to as the dominance ordering on the weight lattice $\Lambda_W$ \end{definition}

The weight lattice $\Lambda_G$ of $G$ is the character lattice $X^*(H)$ of the maximal torus. We have inclusions of lattices: $\Lambda_R \subseteq \Lambda_G \subseteq \Lambda_W$. The quotient of $\Lambda_G$ by $\Lambda_R$ can be identified with the center $Z(G) $ of $G$, which is a finite group. The set of isomorphism classes of finite-dimensional irreducible representations of $G$ are in bijection with points in the cone $\Lambda_G^+ := \Lambda_W^+ \cap \Lambda_G$ of dominant weights for $G$. We denote by $V_\lambda$ the irreducible representation corresponding to $\lambda \in \Lambda_G^+$. Points in the interior of $\Lambda_G^+$ are called regular dominant weights for $G$. In a context where the group $G$ is fixed, we write $\Lambda$ instead of $\Lambda_G$.

Let $P$ be a parabolic subgroup of $G$ with unipotent radical $N_P = R_{\rm u} (P)$, and let $L$ be the quotient of $P$ by $N$. We denote the corresponding Lie algebras as $\frakp = \Lie(P)$,  $\frakn_P = \Lie(N_P)$, and $\frakl = \Lie(L)$.  Since $\frakl$ is the quotient of $\frakp$ by the  Lie subalgebra $\frakn$, it follows that the $\frakn_P$-coinvariants of any $\frakp$-module define a $\frakl$-module. 

\begin{definition} The functor of parabolic restriction on $U\g\dmod$ with respect to $\frakp$ is given by restricting to $U\frakp\dmod$ and then taking $\frakn_P$-coinvariants:
	\begin{align*} \text{\rm res}_{\frakp} : U\g\dmod & \rightarrow U\frakp\dmod \rightarrow  U\frakl\dmod \\ V & \mapsto (V)_{\frakn_P} = V / \frakn_P V.\end{align*}
\end{definition} 

We index conjugacy classes of parabolic subgroups of $G$ by subsets of $\Delta = \{1, \dots, r\}$. Specifically, for a subset $I$ of $\Delta$, we fix a representative $P_I$ for the conjugacy class corresponding to $I$ determined by the condition that $\Lie(P_I)$ is generated by $\fb$ and the weight subspaces $\g_{-\alpha_i}$ for $i \in I$. Let $P_I^-$ be the opposite parabolic to $P_I$, so that  $\Lie(P_I^-)$ is generated by $\fb^-$ and the weight subspaces $\g_{\alpha_i}$ for $i \in I$. The Levi subgroup corresponding to $I$ is the subgroup $L_I$  of $G$ whose Lie algebra is generated by $\fh$ and   $\g_{\pm \alpha_i}$ for $i \in I$. Let $N_I = R_{\rm u}(P_I)$  and $N_I^- = R_{\rm u} (P_I^-)$ be the unipotent radicals. The quotients of $P_I$ and of $P_I^-$ by each of their unipotent radicals is identified with $L_I$; thus, we have projection maps: $\text{\rm pr} : P_I \rightarrow L_I$ and $\text{\rm pr}^{-} : P_I^- \rightarrow L_I$.

\subsection{The Vinberg semigroup and wonderful compactification}

Next, we present a summary of the construction of the Vinberg semigroup and wonderful compactification. For more details, see  \cite{Evenswonderfulcompactification2008a} and \cite[Section 3]{Ganevwonderfulcompactificationquantum}. Let $\O(G)$ denote the coordinate algebra of $G$.  The Peter-Weyl theorem asserts that  the map of matrix coefficients 
$$\phi: \bigoplus_{\lambda \in \Lambda^+} V_\lambda^* \otimes V_\lambda \stackrel{\sim}{\longrightarrow} \O(G); \qquad f \otimes v \mapsto [g \mapsto f(g\cdot v)]$$ defines an isomorphism of $\Ug\ot\Ug$-modules. Moreover, the subspaces $$ \O(G)_{\leq \lambda}  = \phi \left( \sum_{\mu\leq \lambda} V_\mu^* \otimes V_\mu \right),$$ for $\lambda \in \Lambda$, endow $\O(G)$ with the structure of a $\Lambda$-filtered algebra as in Definition \ref{def:filteredalgebra}, invoking the dominance order on $\Lambda$ from Definition \ref{def:gorder}.  Let $\C[\Lambda]$ denote the group algebra of $\Lambda$ as an abelian group; it is generated by formal variables $z^\lambda$ with relations $z^\lambda z^\mu = z^{\lambda + \mu}$, for $\lambda, \mu \in \Lambda$. 

\begin{definition} The Vinberg semigroup $\Vinb_G$ for $G$ is defined as the spectrum of the Rees algebra for $\O(G)$ with the Peter-Weyl filtration: $$\Vinb_G = \Spec\left( \bigoplus_{\lambda \in \Lambda} \O(G)_{\leq \lambda} z^\lambda \right).$$ \end{definition}

The space $\Vinb_G$ is a semigroup with an action of $G \times G$. In addition, since $\Lambda = X^*(T)$ is the character lattice of the maximal torus $T$ of $G$, the $\Lambda$-grading on the coordinate ring of  $\Vinb_G$ endows $\Vinb_G$ with a $T$-action which commutes with the $G \times G$-action. Let $\C[z^{\alpha_i}] =  \C[z^{\alpha_i}\ | \ i = 1, \dots, r]$ denote the polynomial subalgebra of $\C[\Lambda]$ generated by the elements $z^{\alpha_i}$ for $i \in \Delta$. Let $\A = \Spec \left(\C[z^{\alpha_i}]\right)$, so $\A$ is an $r$-dimensional affine space, and the choice of positive simple roots endows $\A$ with a coordinate system. As in \ref{eqn:maptoRees}, we have an inclusion $\C[z^{\alpha_i}] \hookrightarrow \O(\Vinb_G)$. The induced surjective map \begin{equation} \Vinb_G \rightarrow \A \end{equation} is the abelianization map of \cite{Vinbergreductivealgebraicsemigroups1995}; it is flat and equivariant for the natural actions of $T$. The fiber of $\pi$ over a point away from the coordinate hyperplanes in $\A$ can be identified with $G$. The fiber of $G$ over 0 contains the horocycle space $\Y =\horocycle$ as a Zariski open subset. For any weight $\lambda$, thought of as a character of $T$, one can form the GIT quotient $\Vinb_G \text{ $\!$/$\! \!$/$\!$}_\lambda  T$ of $\Vinb_G$ by $T$ along $\lambda$. 

\begin{defprop} \rm Fix a regular dominant weight $\lambda$. The corresponding GIT quotient of $\VinbG$ is a smooth, projective variety that does not depend (up to isomorphism) on the choice of regular dominant weight. It contains the adjoint group $\Gad$ as a Zariski open subset, and is known as the wonderful compactification\footnote{The adjective `wonderful' is a technical term. A variety is called `wonderful' if it is smooth, connected, complete with a group action such that there is an open orbit; moreover, the complement of this orbit must be a union of irreducible divisors with normal crossings whose partial intersections give the remaining orbits. For more details, see \cite{LunaToutevarietemagnifique1996}.} of $\Gad$, denoted by:
	$$\Gbarad := \Vinb_G \text{ $\!$/$\! \!$/$\!$}_\lambda  T.$$ 
\end{defprop}

Implicit in the above Definition-Proposition is that the  semistable loci of $\VinbG$ with respect to various regular dominant weights coincide. Let  $\Vinb_G^{\text{\rm ss}}$ denote be this common semistable locus. This is a $G \times G$ variety, and let 
$$\pi : \Vinb_G^{\text{\rm ss}} \longrightarrow \A$$
be the restriction of the abelianization map to the semistable locus, which is also flat. The map $\pi$ is also smooth, and its fibers can be identified with certain partial horocycle spaces, which we consider in  Section \ref{sec:parabolic} below. We denote by $\D_\pi$ the sheaf of relative differential operators, and $D(\pi)^\gr$ its category of graded modules.  (As an example, we will consider the case of $G = \SL_2$ in detail in Section \ref{sec:sl2} below.)

\begin{rmk} See \cite[Example 3.2.4]{Briontotalcoordinatering2007} for the relation between the definition of the Vinberg semigroup presented in this section and Vinberg's original definition; the latter appears in \cite{Vinbergreductivealgebraicsemigroups1995}. See \cite[Section D.2.3]{DrinfeldGeometricconstantterm2016} for a Tannakian approach to defining the Vinberg semigroup through its category of representations.  For other definitions of the wonderful compactification, see De Concini and Springer \cite{ConciniCompactificationsymmetricvarieties1999} and Evens and Jones \cite{Evenswonderfulcompactification2008a}. The connection to our chosen definition is based on work of Martens and Thaddeus \cite[Theorem 5.3]{MartensCompactificationsreductivegroups2016}, ultimately stemming from Vinberg's seminal paper \cite{Vinbergreductivealgebraicsemigroups1995}.\end{rmk}

\subsection{Differential operators on $G$}\label{subsec:diffopsG}

Let $D_G = \Gamma(G, \D_G)$ be the algebra of global differential operators on $G$. The algebra $D_G$ is the subalgebra of $\End_\C(\O(G))$ generated by $\O(G)$ (acting by left multiplication) and the derivations $\Derv(\O(G))$. Let $\text{\rm Im}(\mu) \subseteq D_G$ denote the image of the quantum moment map $\mu :\Ug \ot \Ug \rightarrow D_G$ stemming from the action of $G \times G$ on $G$ by left and right multiplication.

\begin{definition}For $\lambda \in \Lambda$, we define the following two subspaces of $D_G$:
	\begin{enumerate}
		\item 	 Let $$(D_G)^\text{\rm PW}_{\leq \lambda} := \O(G)_{\leq \lambda} \cdot \text{\rm Im}(\mu)$$  be the subspace generated by the $\lambda$-th filtered piece $\O(G)_{\leq \lambda}$ of $\O(G)$ and the image of $\mu$.  
		
		\item Let $$(D_G)^\text{\rm deriv}_{\leq \lambda} = \{ p \in D_G \ | \ p( \O(G)_{\leq \lambda^\prime} ) \subseteq \O(G)_{\leq \lambda^\prime + \lambda} \ \text{\rm for all $\lambda^\prime$ in $\Lambda$} \}$$ be the space of endomorphisms of $\O(G)$  in $D_G$ sending $\O(G)_{\leq \lambda^\prime}$ into  $\O(G)_{\leq \lambda + \lambda^\prime}$.
	\end{enumerate}
\end{definition}

\begin{prop}\label{prop:PWfilt} The subspaces  $(D_G)^\text{\rm PW}_{\leq \lambda}$ define a $\Lambda$-filtration on $D_G$ (in the sense of Definition \ref{def:filteredalgebra}), which we refer to as the Peter-Weyl filtration on $D_G$. \end{prop}

\begin{proof} Since the subspaces $\O(G)_{\leq \lambda}$ define a $\Lambda$-filtration of $\O(G)$, it suffices to show that $\text{\rm Im}(\mu) \cdot \O(G)_{\leq \lambda}  =  \O(G)_{\leq \lambda} \cdot \text{\rm Im}(\mu) $. To this end, take $f \in \O(G)_{\leq \lambda}$ and $x \in \g \times \g$. The equivariance of the quantum moment map implies that 
		$$\mu(x)\cdot f  = f \cdot \mu(x) + x \triangleright f.$$ Observe that  $x \triangleright f \in \O(G)_{\leq \lambda}$ and hence belongs to  $\O(G)_{\leq \lambda} \cdot \text{\rm Im}(\mu)$. It follows that $\mu(x) \cdot f \in  \O(G)_{\leq \lambda} \cdot \text{\rm Im}(\mu)$, and so $\text{\rm Im}(\mu) \cdot \O(G)_{\leq \lambda}  \subseteq  \O(G)_{\leq \lambda} \cdot \text{\rm Im}(\mu)$. A similar argument shows the reverse inclusion. \end{proof}

\begin{prop}\label{prop:derivationsfiltration} For any $\lambda \in \Lambda$, we have:
$$(D_G)^\text{\rm PW}_{\leq \lambda} = (D_G)^\text{\rm deriv}_{\leq \lambda} $$
 \end{prop}

\begin{proof} 
	It is straightforward to verify that the image of any element of  $\g \times \g$ under the quantum moment map $\mu: \Ug \ot \Ug \rightarrow D_G$ is a derivation sending  $\O(G)_{\leq \lambda^\prime}$ to  $\O(G)_{\leq \lambda^\prime}$ for any $\lambda^\prime \in \Lambda$. (In the notation of Section \ref{subsec:derviRees}, the image of $\g \times \g$ is contained in $\Derv(\O(G))_{\leq 0}$.) Hence the the image of $\Ug \ot \Ug$ lands in  $(D_G)^\text{\rm deriv}_{\leq 0}$. It follows that $(D_G)^\text{\rm PW}_{\leq \lambda} \subseteq (D_G)^\text{\rm deriv}_{\leq \lambda} $ for any $\lambda \in \Lambda$. 
		
	For the opposite inclusion, let $q : \O(G) \ot \Ug \ot \Ug \twoheadrightarrow D_G$ be the quotient map. Since $q(\O(G)_{\leq \lambda} \ot \Ug \ot \Ug)  = (D_G)^\text{\rm PW}_{\leq \lambda}$, it suffices to show that $$q\inv( (D_G)_{\leq \lambda}^\text{\rm deriv}) \subseteq \O(G)_{\leq \lambda} \ot \Ug \ot \Ug.$$ To this end, let $f \ot x \ot y \in \O(G) \ot \Ug \ot \Ug$ be a nonzero element such that $q(f \ot x \ot y)$ belongs to $(D_G)_{\leq \lambda}^\text{\rm deriv}$ (the case of non-simple tensors follows similarly). It is enough to show that $f \in \O(G)_{\leq \lambda}$. 
	
	First, we may choose $\lambda^\prime \in \Lambda$ such that the action of $x \ot y$ on $V_{\lambda^\prime}^* \ot V_{\lambda^\prime}$ is non-zero. Now, since $\O(G)$ is a $\Lambda$-filtered algebra, we have that $\O(G) = \bigoplus_{\gamma \in Z(G)} \left( \cup_{\nu \in \gamma} \O(G)_{\leq \nu}\right)$. Hence, we can write $f = \sum_{\gamma \in Z(G)} f_\gamma$, where $f_{\gamma} \in \O(G)_{\leq \nu_\gamma}$ for some $\nu_\gamma \in \gamma \subseteq  \Lambda$. Moreover, if $f_\gamma \neq 0$, we select $\nu_\gamma$ to be minimal (so that $f_\gamma$ contains matrix coefficients for the irreducible representation $V_{\nu_\gamma}$). Now, the fact that  $q(f \ot x \ot y)$ belongs to $(D_G)_{\leq \lambda}^\text{\rm deriv}$  implies that $q(f \ot x \ot y) ( \O(G)_{\leq \lambda^\prime}) \subseteq  \O(G)_{\leq \lambda + \lambda^\prime}$. By the choice of $\lambda^\prime$, this implies that $f_\gamma \cdot  \O(G)_{\leq \lambda^\prime} \subseteq  \O(G)_{\leq \lambda + \lambda^\prime}$ for each nonzero $f_\gamma$. The minimality assumption on $\nu_\gamma$ implies that $\nu_\gamma \leq \lambda$. We conclude (c.f. Lemma \ref{lem:parialorder}), that $f = f_{\gamma_0}$ for the unique $\gamma_0 \in Z(G)$ such that $\lambda \in \gamma_0$, and so $$f = f_{\gamma_0} \in \O(G)_{\leq \nu_{\gamma_0}} \subseteq  \O(G)_{\leq \lambda}.$$This concludes the proof. \end{proof}

Henceforth, for $\lambda \in \Lambda$, we write $(D_G)_{\leq \lambda}$ for the subspace $(D_G)^\text{\rm PW}_{\leq \lambda} = (D_G)^\text{\rm deriv}_{\leq \lambda} $ of $D_G$. In Proposition \ref{prop:imageofmu} below we show that the zero-th filtered piece is isomorphic to $\Ug \ot_{Z(g)} (\Ug)^{\text{\rm op}}$, where $Z(\g)$ denotes the center of $\Ug$.  Note that, since $G$ is semisimple,  $(D_G)_{\leq \lambda}$ is zero when $\lambda$ antidominant. 

As in Section \ref{subsec:derviRees}, set:
$$\Derv(\O(G))_{\leq \lambda} = \{ \theta \in \Derv(\O(G)) \ | \ \theta( \O(G)_{\leq \lambda^\prime} ) \subseteq \O(G)_{\leq \lambda^\prime + \lambda} \ \text{\rm for all $\lambda^\prime$ in $\Lambda$} \}.$$ 
Recall that the quotient of the weight lattice $\Lambda = \Lambda_G$ of $G$ by the root lattice $\Lambda_R  = \Z\langle \alpha_1, \dots, \alpha_r \rangle$ is identified with the center $Z(G)$ of $G$. 

\begin{cor}\label{cor:derivfilt}  We have:	$$\Derv(\O(G)) = \bigoplus_{\gamma \in Z(G)}  \left( \bigcup_{\lambda \in \gamma} \Derv(\O(G))_{\leq \lambda } \right).$$
\end{cor}

\begin{proof} By definition, $ \Derv(\O(G))_{\leq \lambda } =  \Derv(\O(G)) \cap (D_G)^\text{\rm deriv}_{\leq \lambda}$. The result  follows from Propositions \ref{prop:PWfilt} and \ref{prop:derivationsfiltration}, together with  the definition of a $\Lambda$-filtration (Definition \ref{def:filteredvs}). 
\end{proof}

\subsection{Localization on the Vinberg semigroup}\label{subsec:localizationvinberg}

The main result of this paper is the following:

\begin{theorem}\label{thm:loc} There is a functor $\Asymp : D(G)\filt \rightarrow D_H(\Y)$ that fits into the following commutative diagram:
	\[\xymatrix{ &  & \U(\g \times \g)\dmod \ar[d]^{\Loc_{\pi}} \ar[lldd]_{\Loc_{\Y}} \ar[rrdd]^{\Loc_G} &  & \\ & & D_H(\pi)  \ar[lld]^{ \zeta }  &   & \\  D_H(\Y) & & & &  D(G)\filt  \ar[llll]_{\Asymp }  \ar[llu]^{\rho} }  \]
	\end{theorem}

\begin{proof} The result is an application of the general ideas developed in Section \ref{sec:relativediff}.  The lattice here is $\Phi = \Lambda$, the linearly independent set that determined the partial order is $\{\beta_i\} = \{\alpha_i\}$, and  the group algebra $\C[\Lambda]$ is the coordinate algebra of the maximal torus $H$ of $G$. The $\Lambda$-filtered algebra in question is $A = \O(G)$ and its spectrum $X = \Spec(A)$ is the semisimple group $G$. We set $\VXcirc = \Vinb_G^{\text{\rm ss}}  \subseteq \VinbG  = \Spec(\Rees(\O(G))) $, and have the smooth map $\pi : \VinbG^{\text{\rm ss}} \to \A$.  Additionally, Corollary \ref{cor:derivfilt} implies that the third assumption at the beginning of Section \ref{sec:loccommute} holds. Finally, the group acting is $K = G \times G$ and $\Vinb_G^{\text{\rm ss}}  \subseteq \VinbG$ is a $G \times G \times H$-subvariety. With all these considerations in place, the result is a restatement of Theorem \ref{prop:loccommutesassocgr}.   \end{proof}

The localization functor $\U(\g \times \g)\dmod \rightarrow D_H(\Y)$ is a doubled version of the Beilinson--Bernstein localization functor (in families). 

\begin{prop}\label{prop:imageofmu} The image of the quantum moment map in either $D_G$ or $D_{\Y}$ is isomorphic to $\Ug \ot_{Z(\g)} (\Ug)^{\text{\rm op}}$. We have that:
	$$(D_G)_{\leq 0} = \Ug \ot_{\fZ} (\Ug)^{\text{\rm op}} \cdot \O(G)_{\leq 0}$$
	$$(D_{\Y})_0 = (\Ug \ot (\Ug)^{\text{\rm op}}) \ot_{\fZ \ot \fZ} U\mathfrak t.$$ 
	In addition, the subalgebra of global differential operators on $\Gbarad$ that preserve the $G \times G$-orbits is precisely $\Ug  \ot_{\fZ} \Ug^{\text{\rm op}}$. \end{prop}

\begin{proof} A result of Borho--Brylinski asserts that the $T$-invariant differential operators on $G/N$ is isomorphic to $\Ug \ot_{\fZ} U\mathfrak t$ \cite[Proposition 8]{BorhoDifferentialoperatorshomogeneous1989}. From this,  one deduces that the $T$-invariant differential operators on the horocycle space are  $(\Ug \ot (\Ug)^{\text{\rm op}}) \ot_{\fZ \ot \fZ} U\mathfrak t$. Here we appeal to the Harish-Chandra isomorphism $\fZ \simeq (U \mathfrak t)^W$. The image of $\Ug \ot \Ug$ is isomorphic to $(\Ug \ot (\Ug)^{\text{\rm op}}) \ot_{\fZ \ot \fZ} \C$, where $\fZ \ot \fZ \rightarrow \C$ is the counit map. Observe that the map $(D_G)_{\leq 0} \rightarrow (D_{\Y})_0$ is injective on the image of $\Ug \ot \Ug$. The first result follows. 
	
	The last result is a standard Hamiltonian reduction argument. Alternatively, if $G$ is of adjoint type, then $\O(G)_{\leq 0} \simeq \C$, and we see that the only differential operators that extend to $\Gbarad$ are those in $(D_G)_{\leq 0} \simeq \Ug \ot_{\fZ} (\Ug)^{\text{\rm op}}$.  \end{proof}

\subsection{Case of $\SL_2$} \label{sec:sl2}

For $G = \SL_2$, the Vinberg semigroup is the semigroup of two by two matrices, which we abbreviate by $\Vinb$ for the remainder of this section, and the map $\pi$ is the determinant map $\det: \Vinb \rightarrow \A^1$. The wonderful compactification of $(\SL_2)^{\text{\rm ad}} = \PSL_2$ is $\P^3$, and the horocycle space is the space of two by two matrices of rank one. The semistable locus of $\Vinb$ consists of nonzero matrices. Write $\O(\Vinb) = \C[a,b,c,d]$ for the algebra of functions on $\Vinb$ and $D_{\Vinb} = \C\langle a, b, c, d, \partial_a, \partial_b, \partial_c, \partial_d \rangle$ for the Weyl algebra of differential operators, so that $[\partial_a, a] = 1$, etc. The scaling action of $H = \C^\times$ on $\Vinb$ induces a grading on $D_{\Vinb}$ with the degrees of $a$, $b$, $c$, and $d$ all equal to 1, and the degrees of $\partial_a$, $\partial_b$, $\partial_c$, and $\partial_d$ all equal to $-1$. Let $D_\pi$ be the subalgebra of differential operators relative to the determinant map; i.e.\ the algebra generated by all derivations annihilating the function $ad-bc$.  Write $E, F, H$ for the usual generators of $\Usl$. 

\begin{definition} Set $\Delta = 1+ H^2 + 2EF + 2FE$ to be the Casimir element of $\Usl$ and $\text{\rm Eu} = 1 + a \partial_1 +b \partial_b + c \partial_c + d \partial_d$ to be the Euler operator in $D_\Vinb$. \end{definition}

\begin{lemma} We have:
\begin{enumerate}
\item The algebra $D_\pi$ is generated as an $\O(\Vinb)$-module by the following derivations: $c \partial_a + d \partial_b, b \partial_a + d \partial_c,  a \partial_a - d \partial_d, b \partial_b - c \partial_c,  a \partial_b + c \partial_d,  a \partial_c + b \partial_d.$

\item The quantum moment map $\mu : \Usl \ot \Usl \rightarrow D_\Vinb$ is given by 
\begin{align*} E \ot 1 &\mapsto  - c \partial_a - d \partial_b  & \qquad 1 \ot E &\mapsto  a \partial_b + c \partial_d \\
F \ot 1 &\mapsto  -a \partial_c - b \partial_d & \qquad 1 \ot F &\mapsto b \partial_a + d \partial_c \\
H \ot 1 &\mapsto - a \partial_a - b \partial_b + c \partial_c + d \partial_d & \qquad 1 \ot H &\mapsto a \partial_a - b \partial_b + c \partial_c -d \partial_d
\end{align*} Moreover, the subalgebra $D_\pi$ of relative differential operators is the $\O(\Vinb)$-submodule of $D_\Vinb$ generated by the image of $\mu$.

\item  The following identities hold in $D_\Vinb$:
\begin{align*} (ad-bc)\mu(1\ot E) &= -a^2 \mu(E \ot1) +c^2\mu(F \ot1) +ac \mu(H \ot1)\\
 (ad-bc)\mu(1 \ot F) &= b^2 \mu(E \ot 1)  - d^2\mu (F\ot 1) - bd\mu(H \ot1)\\
 (ad-bc)\mu(1 \ot H) &= 2ab \mu(E\ot1) -2cd\mu(F \ot1) -(ad+bc)\mu(H \ot1)\\
\mu(\Delta \ot 1) &= \mu(1 \ot \Delta) = \text{\rm Eu}^2 - (ad-bc)(\partial_a \partial_d - \partial_b \partial_c). 
\end{align*}
\end{enumerate}
\end{lemma}

We omit the proof of this lemma, as it is a series of straightforward computations. We write $\O(\SL_2) =\C[a,b,c,d]/(ad-bc =1)$ for the algebra of functions on $\SL_2$. 

\begin{prop} The algebra of differential operators on $\SL_2$ is the quotient of the algebra $\O(\SL_2) \ot \Usl \ot \Usl$ by the relations 
\begin{align*} 1 \ot 1 \ot E &=   -a^2\ot  E \ot 1 +c^2\ot F \ot1 +ac\ot H \ot 1  \\ 1 \ot 1 \ot F &= b^2 \ot E \ot 1 - d^2\ot F \ot 1 + bd\ot H \ot 1 \\ 1 \ot 1 \ot H &= 2ab \ot E\ot 1 -2cd\ot F \ot 1 -(ad+ bc)\ot H \ot1 \end{align*} The quantum moment map is the composition of the inclusion $\Usl \ot \Usl \hookrightarrow \O(\SL_2) \ot \Usl \ot \Usl$ and the quotient map. The image of $\O(\SL_2)_{\leq n} \ot \Usl \ot \Usl$ under the quotient is the $n$-th filtered piece of $D_{\SL_2}$. \end{prop}

\begin{proof} Since $\SL_2$ is the fiber of the determinant map over $1$, we have that its algebra of differential operators is the quotient of $D_\pi$ by the (two-sided) ideal generate by $ad-bc=1$. By the previous proposition, we see that $D_{\SL_2}$ is a quotient of $\O(\SL_2) \ot \Usl \ot \Usl$, and to compute the cross relations, we set $ad-bc$ equal to one in the first three identities of the previous lemma. The remaining claims are straightforward. 
\end{proof}

The subspace $1 \ot \mathfrak{sl}_2 \ot 1$ of $\O(\SL_2) \ot \Usl \ot \Usl$ maps isomorphically onto its image in $D_{\SL_2}$, and coincides with the space of right-invariant vector fields on $\SL_2$. Similarly, elements of $1 \ot 1 \ot \mathfrak{sl}_2$ correspond to left-invariant vector fields. The relations listed in the statement of the proposition can be obtained independently by writing a given left-invariant vector fields in terms of right-invariant ones. 

\newcommand{\Hor}{\mathcal Y}

Let $\Hor = \Hor_{\SL_2}$ denote the horocycle space for $\SL_2$, which can be identified with the subspace of rank one matrices in $\Vinb = \Mat_2$, or with the quotient of $(\C^2 \setminus 0) \times (\C^2 \setminus 0)$ by the diagonal scaling action of $\C^\times$. Its algebra of global functions is the associated graded algebra of $\O(\SL_2)$, namely $\O(\Y) = \C[a,b,c,d]/ (ad-bc)$. Let $D_{\Hor}$ be the algebra of differential operators on $\Hor$. By results established above, we have that: 

\begin{lemma} The algebra $D_\Y$ is the quotient of $\Usl \ot \Usl \ot \O(\Hor)$ by the single relation $$\Delta \ot 1 \ot 1 - 1 \ot \Delta \ot 1.$$ The algebra of $C^\times$-invariant differential operators on the horocycle space $\Hor$ is $$U(\mathfrak{sl}_2 \times \mathfrak{sl}_2) \ot_{\fZ} \C[\text{\rm Eu}].$$  \end{lemma}

Here $\fZ = \C[\Delta_1, \Delta_2]$ is the center of $U(\mathfrak{sl}_2 \times \mathfrak{sl}_2)$, and is a polynomial algebra on the two Casimir elements $\Delta_1$ and $\Delta_2$, while $\Delta_i \mapsto \text{\rm Eu}^2$ in $\C[\text{\rm Eu}]$. Note that the quantum moment map is not surjective onto the degree zero piece.  

\section{Localization and parabolic restriction}\label{sec:parabolic}

We now recall a description of the $ G \times G$-orbits on the wonderful compactification $\Gbarad$. For more details, see \cite{Evenswonderfulcompactification2008a, MartensCompactificationsreductivegroups2016, Vinbergreductivealgebraicsemigroups1995}. There is a bijection between $G \times G$-obits $\X_I \subseteq \Gbarad$ and subsets $I \subseteq \Delta$ of (the indexing set for) the set of positive simple roots, with the property that $\X_I$ is contained in the orbit closure $\overline{\X_J}$ if and only if $I \subseteq J$.  In the extreme cases, we have $\X_\Delta = \Gad$ and $\X_\emptyset = G/B \times B^- \backslash G$. 
 
Fix a subset $I \subseteq \Delta$, and let $P_I$ and $L_I$ be the corresponding parabolic and Levi subgroups (see Section \ref{sec:notation} above). We compose the projection maps: $\text{\rm pr}: P_I \rightarrow L_I$ and $\text{\rm pr} : P_I^- \rightarrow L_I$, to obtain a map valued in $L_I^{\text{ad}} = L_I/ Z(L_I)$. Since $\Lad$ is a reductive group of adjoint type, it has a wonderful compactification $\overline{\Lad}$.  There is a point in $\X_I$ whose stabilizer is the subgroup
	$$P_I \times_{L_I^{\text{ad}}} P_I^- = \{ (g,h) \in P_I \times P_I^- \ | \ \text{\rm pr}(g) \left( \text{\rm pr}^-(h)\right)\inv \in Z(L_I)\}.$$
In addition, there are $G \times G$-equivariant fibrations:
	\[ \xymatrix{ \Lad \ar[r]  & \X_I \ar[d] \\ & G/P_I \times P_I^- \backslash G  }  \qquad \xymatrix{ \overline{\Lad} \ar[r]  & \overline{\X_I} \ar[d] \\ & G/P_I\times P_I^- \backslash G  }. \]
	
Since $L_I$ normalizes the unipotent radical $N_I = R_{\rm u}(P_I)$, there is a right $L_I$-action on $G/N_I$ and a left $L_I$ action on $N_I^-\backslash G$. These combine to give an `internal'  $L_I \times L_I$ on $G/N_I \times N_I^-\backslash G$, which commutes with the `external' action of $G \times G$. We consider the balanced product, i.e., the quotient of $G/N_I \times N_I^- \backslash G$ by the diagonal  $(L_I)_\Delta$ of $ L_I \times L_I$:

\begin{definition} For $I \subseteq \Delta$, the corresponding partial horocycle space is  the  quotient of $G/N_I \times N_I^- \backslash  G$ by the action of $L_I$:  $$\Y_I = {G/N_I \times_{L_I} N_I^- \backslash G}$$
\end{definition}

Since $Z(L_I) \times Z(L_I)$ normalizes $(L_I)_\Delta$, we obtain a free action of $(Z(L_I) \times Z(L_I))/(Z(L_I) \times Z(L_I)\cap (L_I)_\Delta) \simeq Z(L_I)$ on  $\Y_I$. The quotient by this action is precisely $\X_I$. We obtain $G \times G$-equivariant fibrations:
	\[ \xymatrix{ Z(L_I) \ar[r]  & \Y_I \ar[d] \\ & \X_I }  \qquad \qquad \xymatrix{ {L}_I \ar[r]  & \Y_I \ar[d]^{q_I} \\ & G/P_I\times P_I^- \backslash G }. \]

The partial horocycle spaces appear in the Vinberg semigroup, as we now explain. Let $e_I$ be the point in $\A^r$ whose $i$th coordinate is 1 if $i \in I$ and zero otherwise. The group $G \times G$ acts transitively on each fibers of the map $\pi : \Vinb_G^{\text{\rm ss}} \rightarrow \A^r$, and  ${\pi}\inv(e_I)$ is identified with $\Y_I$. Thus we have an inclusion $i_I : \Y_I \hookrightarrow \Vinb_G^{\text{\rm ss}} $. While the action of $H$ on $\Vinb_G^{\text{\rm ss}} $ does not preserve $\Y_I$, the action of the subgroup $Z(L_I) \subseteq H$ does, and coincides with the action defined above. Hence we obtain a restriction  functor:
	$$i_I^* : D_H(\Vinb_G^\text{\rm ss}) \rightarrow D_{Z(L_I)} (\Y_I)$$
In fact, the orbit $\X_I \subseteq \Gbarad$ is the GIT quotient of the fiber of the abelianization map $\tilde{\pi} : \Vinb_G \to \A^r$ over $e_I$ by the action of $Z(L_I) \subseteq H$, and $\Y_I = \pi\inv(e_I) \subseteq \tilde{\pi}\inv(e_I)$ is the semistable locus for this action. The following result is a special case of the general constructions of Section \ref{sec:relativediff} (see also Section \ref{subsec:localizationvinberg}): 

\begin{prop}\label{prop:asympI} There is a `parabolic asymptotics' functor $ \Asymp_I : D(G)\filt \to D_{Z(L_I)}(\Y_I)$
that fits into the following commutative diagram:
\[\xymatrix{ &  & \U(\g \times \g)\dmod \ar[d]^{\Loc_{\Vinb_{G}}} \ar[lldd]_{\Loc_{\Y_I}} \ar[rrdd]^{\Loc_G} &  & \\ & & D_{H}(\pi)  \ar[lld]^{ i_I^* }  &   & \\  D_{Z(L_I)}(\Y_I) & & & &  D(G)\filt  \ar[llu]_{\rho} \ar[llll]^{\Asymp_I } }  \]
\end{prop}

We fix a point $y \in \Y_I$ and set $x =q_I(y)  \in G/P_I \times P_I^-\backslash G$ to be its image under  the fibration $q_I : \Y_I\rightarrow G/P_I \times P_I^- \backslash G$ (see Section \ref{sec:notation}). The choice of $y \in \Y_I$ identifies the fiber $q_I\inv(x)$ with $L_I$, and we have an inclusion:
$$i_y : L_I \hookrightarrow \Y_I$$
On the other hand, the points of the partial flag variety $G/P_I$ are in bijection with the conjugates of $P_I$ (i.e., every parabolic subgroup is its own normalizer), and it follows that $x \in G/P_I \times P_I^- \backslash G$ corresponds to a pair $(P, P^\prime)$ of parabolic subgroups, where $P$ is conjugate to $P_I$ and $P^\prime$ is conjugate to $P_I^-$. We set $\frakp = \Lie(P)$ and $\frakp^\prime = \Lie(P^\prime)$ and note that the quotients of $P$ and $P^\prime$ by each of their unipotent radicals are canonically identified with the Levi subgroup $L_I$. 

\begin{theorem}\label{thm:parabolic} The following diagram commutes:
	\[ \xymatrix{  U\g \ot U\g\dmod \ar[rrrr]^{\text{\rm res}_{\frakp} \ot \text{\rm res}_{\frakp^\prime} } \ar[d]^{ \Loc_{G} } & &   & &  U\frakl_I \ot U\frakl_I\dmod \ar[d]^{\Loc_{L}}    \\ D(G )\filt  \ar[rrrr]^{i_y^* \circ \Asymp_I}  & &  & & D_{Z(L_I)} (L_I) },\]
	where the vertical functors are  matrix coefficients localization.  \end{theorem}

In other words, the functor of matrix coefficients localization transforms parabolic restriction into parabolic asymptotics. 

\begin{proof} By Proposition \ref{prop:asympI}, the composition of matrix coefficient localization and $\Asymp_I$ is the same as localization onto $\Y_I$. Let $\tilde \Y_I := \frac{G/N_I \times G/ N_I^- }{L_I} $ be the quotient of $G/N_I \times G/N_I^-$ by the right diagonal action of $L_I$. Applying the inverse on $G$ in the second factor, we obtain an isomorphism
$\phi: \Y_I\stackrel{\sim}{\longrightarrow}  \tilde \Y_I.$ We fix a point $(g,h) \in G \times G$ whose image under the quotient map  $G \times G \twoheadrightarrow \tilde \Y_I$ is equal to $\tilde y := \phi(y)$, and set $i_{\tilde y} := \phi\circ i_y$. Then we have that  $P = gP_I g\inv$ and $P^\prime = hP_I^- h\inv$, the inclusion $i_{\tilde y}$ is given by $\ell_0 \mapsto [g\ell_0,h]$. Thus, it suffices to show that the following diagram commutes:
\[ \xymatrix{  U\g \ot U\g\dmod \ar[rrrr]^{\text{\rm res}_{\frakp} \ot \text{\rm res}_{\frakp^\prime} } \ar[d]^{ \Loc_{\tilde \Y_I } } & &   & &  U\frakl_I \ot U\frakl_I\dmod \ar[d]^{\Loc_{L}}    \\ D_{Z(L_I)} (\Y_I )  \ar[rrrr]^{i_{\tilde y}^* }  & &  & & D_{Z(L_I)} (L_I) }.\]

Let $\underline{\fraks}_L$ be the kernel of the anchor map $(\frakl_I \times \frakl_I) \ot \O_{L_I} \rightarrow D_{L_I} $ for the multiplication action of $L_I \times L_I$ on $L_I$, and  $\underline{\fraks}_{\tilde \Y}$  the kernel of the anchor map $(\g \times \g) \ot \O_{\Y_I} \rightarrow D_{\tilde \Y_I} $. Let $V \ot V^\prime$ be a  $U\g \ot U\g$-module.  By Lemma \ref{lem:coinvariants},  going right and then down in the diagram, we obtain the $\underline{\fraks}_{L}$-coinvariants   of the $(\frakl_I \times \frakl_I) \ot \O_{L_I}$-module sheaf $(\text{\rm res}_\frakp (V) \ot \text{ \rm res}_{\frakp^\prime}(V^\prime)) \ot \O_{L_I}$. On the other hand, going down then right, we obtain the $i_{\tilde y}^* \left( \underline{\fraks}_{\tilde \Y}\right)$-coinvariants   of the $i_{\tilde y}^* \left( (\g \times \g) \ot \O_{\tilde \Y_I} \right)$-module sheaf $( V \ot V^\prime) \ot \O_{\tilde \Y_I}$. Thus, by the definition of parabolic restriction,  it suffices to show that 
$$ i_{\tilde y}^* \left(\underline{\fraks}_{\tilde \Y} \right) = (\frakn_I \times \frakn_I^-) \ot \underline{\fraks}_L$$
as coherent sheaves on $L_I$. To see this, observe that the stabilizer in $G \times G$ of $i_{\tilde y}(\ell_0)$ is given by:
\begin{align*}
\{ (g \ell_0 \ell n \ell_0 \inv g\inv, h \ell m h\inv  ) \in G \times G \ | \ \ell \in L_I, n \in N_I, m \in N_I^-  \} \\
 = \{ (g \ell_0 \ell \ell_0\inv g\inv, h \ell h \inv  )  \ : \ \ell \in L_I \} \left(g \ell_0 N_I \ell_0\inv g\inv \times h (N_I^-) h\inv\right),
\end{align*}
i.e.\ the $(g, h)$-conjugate of $\{ (\ell_0 \ell \ell_0 \inv, \ell) \ : \ \ell \in L_I \} \left( \ell_0 N_I \ell_0\inv \times N_I^-\right)$. Since  the inclusion $i_{\tilde y} : L_I \to \tilde \Y_I$ is given by $\ell_0 \mapsto [g\ell_0,h]$, it follows that the free coherent sheaf $(\frakn_I \times \frakn_I^-) \ot \O_{L_I}$ includes into $i_{\tilde y}^* \left(\underline{\fraks}_{\tilde \Y}\right)$. Meanwhile, the stabilizer in $L \times L$ at $\ell_0 \in L$ is $\{ (\ell_0 \ell \ell_0\inv, \ell ) \ : \ \ell \in L  \}$. Thus, $\underline{\fraks}_{L}$ also includes into $\underline{\fraks}_{\tilde \Y}$, and together these generate $i_{\tilde y}^* \left(\underline{\fraks}_{\tilde \Y}\right)$.  \end{proof} 

\section{Relation to Verdier specialization}\label{sec:Verdier}

\subsection{General set-up}

\newcommand{\Xbar}{\overline{X}} 

Let ${\Xbar}$ be a smooth variety equipped with $r$ smooth divisors $Z_1, \dots, Z_r$ with normal crossings. We assume that the intersection $W = \bigcap_i Z_i$ is smooth. Let  $Z = \bigcup_i Z_i$ be the union of the divisors $Z_i$, and let $ X = {\Xbar} \setminus  Z$ be the complement of $Z$. Let $\Lambda = \Z^r$ and $T = (\C^\times)^r$ so that $\Lambda$ is identified with the character lattice of $T$. Let $N_W({\Xbar})$ be the normal bundle of $W$ in ${\Xbar}$, and denote by $N_W(Z)$ the union of the normal bundles $\bigcup_i N_W(Z_i)$ of $W$ in $Z_i$. We have that the complement $X_0 = N_W({\Xbar}) \setminus N_W(Z)$ is a smooth subvariety of $N_W(\Xbar).$  Consider the following subsheaves of $\O_{\Xbar}$, for $\mathbf k = (k_1, \dots, k_r) \in \Z^r$:
 $$(\O_{\Xbar})_{\leq \mathbf{k}} := \prod_i \cI_{Z_i}^{-k_i} = \cI_{Z_1}^{-k_1}\cI_{Z_2}^{-k_2} \cdots \cI_{Z_r}^{-k_r},$$ where $\cI_{Z_i}$ is the ideal sheaf of the divisor $Z_i$, and $\cI_{Z_i}^{-k_i} = \O_{\Xbar}$ if $k_i > 0$. These define a $\Z^r$-filtration on $\O_{\Xbar}$.

\begin{definition} Define the $V$-filtration on $\D_{\Xbar}$  as
$$\Gamma(V, (\D_{\Xbar})_{\leq \mathbf n})= \{ P \in \Gamma(V, D_{\Xbar}) \ | \ P(\Gamma(V, (\O_{\Xbar})_{\leq \mathbf k})) \subseteq \Gamma(V, (\O_{\Xbar})_{\leq \mathbf{k}-\mathbf{n}}) \ \text{\rm{for all $\mathbf k \in \Z^r$}} \},$$
for $\mathbf n \in \Z^r$. Define a filtration on $j_* \D_{X}$ whose $\mathbf n$-th piece is defined as the set of $P  \in \Gamma(V, j_* \D_X)$  such that upon restriction to $V^\prime \cap X$ for any open subset $V^\prime$ of $V$, the operator $P$ takes sections of $(\O_{\Xbar})_{\leq \mathbf k}$ to sections of $(\O_{\Xbar})_{\leq \mathbf{k}- \mathbf{n}}.$\end{definition}

\begin{lemma} The natural map of restriction $\phi: \D_{\Xbar} \rightarrow j_* \D_X$ respects the filtrations. Consequently, we have a pullback functor: $j_* \D_X\dmod\filt \rightarrow \D_{\Xbar}\dmod\filt$. \end{lemma}

\begin{proof} Let $P \in \Gamma(V, (\D_X)_{\leq \mathbf{n}})$. Then for all $V^\prime \subseteq V$, the following diagram commutes:
\[ \xymatrix{ \Gamma(V^\prime, (\O_{\Xbar})_{\leq \mathbf{k}}) \ar[r]^{P}  \ar[d] & \Gamma(V^\prime, (\O_{\Xbar})_{\leq \mathbf{k} -\mathbf{n}})\ar[d] \\ \Gamma(V^\prime \cap X, \O_{\Xbar}) \ar[r]^{\phi(P)} & \Gamma(V^\prime \cap X, \O_{\Xbar} ) }\] \end{proof}

Standard results on $V$-filtrations (e.g., \cite{SabbahMHMProject}) imply the following: 

\begin{lemma} Let  $\nu : N_Z(\Xbar) \rightarrow Z$ be the normal bundle. 
\begin{itemize}
\item The associated graded of $\D_{\Xbar}$ is supported on $Z$ and is identified with $\nu_* \D_{N_Z(\Xbar)}$.
\item The associated graded of $j_* \D_X$ is supported on $Z$ and is identified with $\nu_* \D_{X_0}$.
\item There are functors of Verdier specialization:
$$ D_{\text{\rm rh}} ( \Xbar ) \stackrel{\text{\rm Sp}}{\longrightarrow} D(N_W(\Xbar)) \qquad \qquad  D_{\text{\rm rh}} ( X ) \stackrel{\text{\rm Sp}^\circ}{\longrightarrow} D(X_0).$$
\end{itemize}
 \end{lemma}

\subsection{Case of the wonderful compactification}

We apply the above set-up to the case of the wonderful compactification.  We assume for simplicity that $G$ is adjoint, and let $\Gbar$ denote the wonderful compactification. Let $j : G \hookrightarrow \Gbar$ be the inclusion. We have boundary divisors $Z_i$ for $i = 1, \dots, r$ $=$ $\text{\rm rank} (G)$.  In the notation from above, $\Xbar = \Gbar$, $X = G$, $W = G/B \times B^- \backslash G$, and $X_0 = \mathcal Y = \horocycle$. 

\begin{prop}\label{prop:vfiltration} On the level of global sections, the $V$-filtration on $j_* D_{G}$ coincides with the matrix coefficients filtration on $D_{G}$. \end{prop}

\begin{proof} The multi-Rees space of $\O_G$ with the matrix coefficients filtration is the Vinberg semigroup, and $X_0 = \horocycle$, and hence the $\Lambda$-filtration on $j_* \O_G$ by order of pole along the $Z_i$'s coincides with the matrix coefficients filtration.\end{proof}

Let $M$ be a regular holonomic $D$-module on $G$, so that $j_* M$ is a regular holonomic $D$-module on $\Gbar$. Note that:
\begin{itemize}
	\item  The space of global sections $\Gamma(\Gbar, j_*M) =\Gamma(G, M)$ has an action of $D_G$. 
	\item  $M$ on $\Gbar$ has a Kashiwara--Malgrange filtration based on order of vanishing along any of the smooth divisors $Z_i$. Taking all these filtrations at once, we obtain a $\Lambda$-filtration on the global sections $\Gamma(\Gbar, j_* M) = \Gamma(G, M)$.  
\end{itemize}

\begin{prop} The Kashiwara--Malgrange filtration on the global sections of $M$ is compatible with the matrix coefficients filtration on $D_G$.  \end{prop}

Thus, given any regular holonomic module $M$, we can lift $M$ to an object in $D(G)\filt$ whose associated graded coincides with the Verdier specialization of $M$. 

\[ \xymatrix{   M = \bigcup_{\lambda \in \Lambda} M_{\leq \lambda}  & D(G)\filt \ar[d]_{\text{\rm forget}} \ar[rr]^{\text{\rm ass. gr.}} & & D_H(\mathcal Y) \ar[dd]^{\text{\rm forget}}\\
& D(G) & & \\
M \ar@{|~>}[uu] & D_\text{\rm rh}(G) \ar[u]^{\text{\rm forget}} \ar[rr]^{\text{\rm Sp}^\circ} & & D(\mathcal Y)
}\]

Let $ D_\text{\rm rh}^{\text{$Z$-l.f.}}(G)$ denote the category of regular holonomic $D$-modules on $G$ with the property that  the action of the center $Z(U\g)$ coming from the action of $\Gad$ on itself by left translations is locally finite.  Then, by \cite[Section 6]{BezrukavnikovCharactermodulesDrinfeld2012}, any object in $ D_\text{\rm rh}^{\text{$Z$-l.f.}}(G)$ specializes to a $H$-monodromic $D$-module on $\mathcal Y$. We can summarize the relation between our associated graded functor and the functor appearing in \cite{BezrukavnikovCharactermodulesDrinfeld2012} in the following diagram:

\[ \xymatrix{   M = \bigcup_{\lambda \in \Lambda} M_{\leq \lambda}  & D(G)\filt \ar[d]_{\text{\rm forget}} \ar[rrd]^{\text{ass. gr. }} & & \\
& D(G) & & D_H(\mathcal Y) \\
M \ar@{|~>}[uu] & D_\text{\rm rh}^{\text{$Z$-l.f.}}(G) \ar[u]^{\text{\rm forget}} \ar[rru]^{\text{\rm Sp}^\circ} & &
}\]

\subsection{Harish-Chandra bimodules}

Let $U(\g_\Delta) \hookrightarrow U\g \ot U\g$ be the inclusion of the diagonal, and $Z\g \ot Z\g \hookrightarrow U\g \ot U\g$ be the inclusion of the center. We may regard any $D$-module on $G$ as a module for $U\g \ot U\g$, $Z\g$, and $U(\g_\Delta)$ via the map $U\g \ot U\g \rightarrow D_G$. 

\begin{definition} A finitely-generated $U\g \ot U\g$-module $V$ is called a Harish-Chandra bimodule if it is  locally finite as a $U(\g_\Delta)$-module, and locally finite as a $Z\g \ot Z\g$-module. We denote the resulting category by $\mathcal{HC}$.  \end{definition}	

\begin{theorem} [\cite{GinsburgAdmissiblemodulessymmetric1989}] If $V$ is a Harish-Chandra bimodule, then its localization $\Loc_G(V)$ is a regular holonomic $D$-module on $G$, and locally finite for the action of $Z(U\g)$ coming from the action of $G$ on itself by left translations. \end{theorem}

It follows that our associated graded functor matches with Verdier specialization for localizations of Harish-Chandra bimodules, and  we have the following commutative diagram.

\[ \xymatrix{ &  & & D(G)\filt \ar[drr]^{\text{\rm ass. gr.}}  \ar[d]^{\text{\rm forget}} & & \\
\mathcal{HC} \ar@{^{(}->}[rr] \ar[drrr] & &  U\g\ot U\g\dmod  \ar[r]^{\qquad \Loc_G} \ar[ru] &   D(G)  & & D_H(\mathcal Y)\\
& &  & D_{\text{\rm rh}}^{\text{$Z$-l.f.}}(G)\ar[u]_{\text{\rm forget}} \ar[urr]_{\text{\rm Sp}^\circ} &
 }\]
 
 \section{The multi-temporal wave equation}\label{wave section}
We now describe an observation about the relation of localization and the multi-temporal wave equation of Semenov-Tian-Shansky. For simplicity, and to ease notation, in this section we assume that $G$ is of adjoint type. 

First, we can extend the localization construction to allow differentiation along the times of the Vinberg degeneration. We define the following algebras:
$$UU:=U\fg\ot_\fZ U\fg\ot U\fh\longrightarrow \wt{UU}:=U\fg\ot_\fZ U\fg\ot_\fZ U\fh.$$
The Vinberg semigroup $\Vinb_G$ carries an action of $G \times G \times H$, which induces an infinitesimal action of $UU$ on $\Vinb_G$, and on the horocycle space $\Y$, this action factors through $\wt{UU}$. Given $\lambda\in \fh^\ast$ and the corresponding one-dimensional representation $\C_\lambda$ of $U\fh$,  there is a localization functor
$$\Loc_{\Vinb_G, \lambda} : U\g \ot_\fZ U\g\dmod \to D^{\text{\rm log}}(\Vinb_G)$$
$$M\mapsto (M\ot \C_\lambda)\ot_{UU} \D^{\text{log}}_{\Vinb_G}.$$
On the group locus $G\times H\subset \Vinb_G$, this functor couples the relative localization to a rank one flat connection on $H$ with monodromy $\exp(\lambda)$. At the other extreme, on  $\Y$ this localization picks out the $\lambda$-monodromic Beilinson--Bernstein localization, as a twisted $D$-module on $G/B\times B^-\backslash G$, out of the full horocycle localization. This is the counterpart of taking the $\lambda$-homogeneous component of the asymptotics of a matrix coefficient (i.e., the $\lambda$-contribution to the Mellin transform of the horocycle transform). 

A more interesting construction couples the relative systems with the (complexified) multitemporal wave equation for symmetric spaces\footnote{Here we are considering the group case -- to get the multi-temporal wave equations for other symmetric spaces $G/K$ one replaces  $G$ by the Vinberg degenerations of $G/K$ to $G/MN$, with times given by $A$, as in~\cite{AbeJacquetFunctorConcini2015, ChenformulageometricJacquet2017}.} of Semenov-Tian-Shansky~\cite{Semenov-Tian-ShanskyHarmonicanalysisRiemannian1976,PhillipsScatteringtheorysymmetric1993,HelgasonIntegralgeometrymultitemporal1998} on $G\times H$.

\begin{definition}
The Vinberg wave localization is the functor 
$$\wt{\Loc}_{\Vinb_G}: U\fg\ot_\fZ U\fg\dmod \longrightarrow \D_H^{\text{log}}(\Vinb_G)$$ defined by 
$$M\mapsto (M\ot_\fZ U\fh) \ot_{UU} \D^{\text{log}}_{\Vinb_G}$$
\end{definition}

\begin{rmk} The wave localization defines a log $D$-module on $\Vinb_G$, whose restriction to the horocycle space $\Y$ agrees with that of the relative localization functor $\Loc_{\Vinb_G}$, i.e. usual Beilinson--Bernstein localization, tensoring $M$ only over $U\fg\ot_\fZ U\fg$. On the group locus, which is identified with $G\times H$, the wave localization is typically $|W|$ times larger than the relative localization functor, where  $W$ is the Weyl group of $G$. The reason for this is that we induce to the quotient $UU\to \wt{UU}$ from the subalgebra $U\fg\ot_\fZ U\fg$ acting fiberwise, i.e., we couple to a rank $|W|$ flat connection in the transverse directions. 
\end{rmk}

Recall that the multi-temporal wave equation is the system of equations on the product $\fa\times G/K$ of a (real!) symmetric space with the maximally split Cartan, given by 
\begin{equation}\label{wave}
\{ HC(z)\cdot f - z\cdot f=0,\hskip.3in z\in D(G/K)^G\}
\end{equation}
where we use the Harish-Chandra isomorphism of $G$-invariant differential operators on the symmetric space with $\Sym(\fa)^W$, considered as constant coefficient differential operators on the Lie algebra $\fa$.

The wave localization of the free module $Wv=\wt{\Loc}_{\Vinb_G}(U\fg\ot_\fZ U\fg)$ restricted to the group locus $G\times H\subset \Vinb_G$ precisely recovers the (complexified) wave equation. Namely, 
a solution of $Wv|_{G\times H}$ in a $D$-module $\cF$ is a section $f$ of $\cF$ satisfying the system of equations~\ref{wave} with $\fZ$ playing the role of $D(G/K)^G$ through the Harish-Chandra isomorphism
$$HC:\fZ=D_G^{G\times G}\longrightarrow U\fh^W\subset U\fh=D_H^H$$
taking elements of the center to constant coefficient differential operators on the Cartan.
Here the time variables are the coordinates of the Lie algebra $\fh$, i.e., we've written the wave equation in exponentiated time. Approaching the complement $\A^r\setminus H$ (in particular the origin) is thus considering infinite time behavior of the equation, i.e., the scattering data.

More generally, the wave localizations $\wt{\Loc}_{\Vinb_G}(M)$ are systems of equations combining the matrix coefficient $D$-modules on $G$ with the wave equation along $H$. The restrictions of these localizations to the asymptotic cone and other strata in $\Vinb_G$ are describing the scattering data of the system, i.e., the differential equations satisfied by scattering data of solutions.

\section{Application: classical asymptotics of matrix elements of admissible representations}\label{real groups section}
In this section we explain how some classical results on the asymptotics of matrix coefficients of admissible representations (for which we follow Casselman and Mili\v{c}i\'{c}~\cite{CasselmanAsymptoticbehaviormatrix1982}) can be recovered from the perspectives put forward in this paper. 

We fix a real reductive group $G_\R$ with an Iwasawa decomposition $G_\R=K_\R A_\R N_\R$ (we denote by $G,K,A$ etc.\ the complexifications of the corresponding real groups). We wish to describe asymptotics of matrix elements on $G_\R$ using the wonderful compactification of $G$. Recall that the Langlands decomposition of the corresponding minimal parabolic subgroup is given by $P_\R=M_\R A_\R N_\R$.  Let $I$ be the subset of positive simple roots corresponding to $P$. As we recalled in Section \ref{sec:parabolic}, the corresponding $G\times G$ orbit closure $\overline{\X}_I  \subset \Gbar$ in the wonderful compactification fibers over $G/P\times P^-\backslash G$. The Iwasawa decomposition defines a real point in the partial flag variety, $\{N,N^-\}\in G/P\times P^-\backslash G$. Moreover, the fiber $F_N$ of $\overline{\X}_I$ over $\{N,N^-\}$ is identified canonically with $M$. The fiber $\wt{F}_N$ over $\{N,N^-\}$ in the deleted normal cone $\Y_I$ of the orbit closure $\overline{\X}_I$ is identified non-canonically with $L=MA$ -- we consider it as an $L$-torsor, which is naturally an $A$-torsor over $F_N\simeq M$.

The Cartan decomposition $G_\R=K_\R A_\R K_\R$ reduces the study of noncompact directions in $G_\R$, and hence of asymptotics, to the torus $A_\R$ (or via the exponential ma $$\exp:\R^d\simeq \fa_\R\stackrel{\sim}{\longrightarrow} A_\R$$ to the Lie algebra $\fa_\R$), and in fact to the negative Weyl chamber $A_\R^-\subset A_\R$. More specifically, the choice of Iwasawa decomposition determines a distinguished point $x \in F_N \subset \overline{\X}_I,$ corresponding to $ \{1\}\in M$ under the identification $M \simeq F_N$. The point $x$  lies in the  closure of $A_\R$ in the wonderful compactification $\Gbar$. We study matrix coefficients by their expansion around the point $x$. It is important to note that a sufficiently small neighborhood of $x$ at infinity  in $A$ stays in the regular semisimple locus (i.e., we are going off to infinity in a generic direction ``away from the walls'' in $A_\R$). 

Let $V_{\rm top}$ denote an admissible representation of $G_\R$, and $V$ the corresponding $(\g, K)$-module consisting of the $K_\R$-finite vectors in $V_{\rm top}$. Given a $K$-finite vector $v\in V\subset V_{\rm top}$ and a $K$-finite covector $v'\in V'\subset V_{\rm top}^*$, we consider the matrix coefficient 
$$ G_\R \ni g \mapsto m_{v,v'}(g)= \langle v', g\cdot v\rangle.$$ 
The $K_\R$-finiteness of $v,v'$ allows us to consider $m_{v,v'}$ instead as a smooth section of a vector bundle on $K_\R \backslash G_\R /K_\R$. In the  terminology of~\cite{CasselmanAsymptoticbehaviormatrix1982}, this section is   a {\em $\tau$-spherical function}, where $\tau$ is a representation of $K_\R\times K_\R$ carried by $v\ot v'$. The asymptotics of $\tau$-spherical functions, and thus of $K$-finite matrix coefficients, reduces to the study of their restriction as vector-valued functions to the negative Weyl chamber $A_\R^-\subset A_\R$.
 
We approach the matrix coefficient $m_{v,v'}$ only through the differential equations it satisfies. First, the localization of the $(\fg,K)$-bimodule $V\ot V'$ defines a $K$-biequivariant $D$-module on $G$, i.e., a $D$-module $\cM_{v,v'}$ on $K\backslash G/K$. The matrix coefficient $m_{v,v'}$ itself defines a smooth solution of $\cM_{v,v'}$ along the real locus $K_\R\backslash G_\R/K_\R$. 
By admissibility, the center $\fZ \subset \Ug$ acts on $V$ through a finite dimensional quotient. This implies that $\cM_{v,v'}$ is an admissible $D$-modules on $G/K$, in particular regular holonomic and a local system on the regular semisimple locus, from which the real analyticity of $m_{v,v'}$ on $A_\R^-$ follows.

Deligne's theory of regular singular equations~\cite{DeligneEquationsDifferentiellesPoints1970}, as explained by Casselman and Mili\v{c}i\'{c} in the Appendix to \cite{CasselmanAsymptoticbehaviormatrix1982}, provides an explicit form for the asymptotics of their solutions. Namely we are considering a regular singular $D$-module on the torus $A$ which is lisse (a flat connection) in a neighborhood of a chosen point at infinity -- the point $0$ in suitable coordinates $A\hookrightarrow \C^r$ (given by a chosen basis of weights). Then we find that solutions in the neighborhood of  $A_\R^-$ can be written as a finite sum\footnote{The sum is easily normalized so as to make it unique.} of expressions \begin{equation}\label{Deligne asymptotics}
 F_{\lambda,m} z^\lambda \log^m z
 \end{equation}  
 where $z^\lambda$ ($z\in A$) are the characters of $A$, i.e., complex exponentials $e^{\lambda t}$ as functions of $t\in \mathfrak a$, $m$ is a positive integer, and the coefficients $F$ are regular holomorphic functions.
 Casselman and Mili\v{c}i\'{c} then deduce Harish-Chandra's results that the crucial growth properties (temperedness and $p$-integrability modulo center) of the matrix coefficient -- and, varying $v,v'$, of the representation $V$ itself -- are completely controlled by the leading exponent $\lambda$ in the above expression (with respect to the dominance order).
 
We would thus like to recover the information of the exponents $\lambda_i$ and log multiplicities $m_i$ from representation theory. We first observe that this information (for a regular singular $D$-module on the torus $A$, which is a local system in the deleted neighborhood of a point at infinity) is contained in the Verdier specialization of the $D$-module to the deleted tangent space at the point. In other words, we linearize the $D$-module around the chosen point, retaining the local monodromy information. This specialization is a finite rank $A$-monodromic $D$-module on an $A$-torsor, which is equivalent to a finitely supported module over $U\fa=\C[\fa^\ast]$. The set-theoretic support and multiplicities of this module reproduce the $\lambda_i$ (complex characters of $A_\R$) and $m_i$. 

This specialization on $A$ is part of the data of the specialization of our matrix coefficient $D$-module at infinity, which we described using parabolic restriction. Namely, the specialization of $\cM_{v,v'}$ along the stratum $\overline{\X}_I$ is a $K\times K$-equivariant $A$-monodromic sheaf on $\Y_I$. Its restriction to the fiber $\wt{F}_N\simeq L=MA$ is an $M\times M$-equivariant $A$-monodromic sheaf $\cM_{v,v',N}$, i.e., equivalent to an $A$-monodromic sheaf on $A$ (or rather an $A$-torsor) valued in representations of $M$. By Theorem~\ref{thm:parabolic}, one can identify this sheaf with the $M$-equivariant matrix coefficient $D$-module of the parabolic restriction of $V\ot V'$, which is an admissible $(\fl=\fa\oplus\fm,M)$-bimodule. Specifically, this parabolic restriction is given by diagonal $\fa$-coinvariants on the zeroth $\fn\oplus \fn^-$-homology of $V\ot V'$. Thus, $\cM_{v,v',N}$ is a union of finite dimensional $(\fa,M)$-submodules, or, dually, of finitely supported coherent sheaves on $\fa^\ast$ valued in representations of $M$. Moreover, eventually (e.g. filtering by highest weights of $K$-representations) this union will contain the image of our $K\times K$-finite vector $v\ot v'$. In particular we have established the following:

\begin{theorem}\label{asymptotics theorem}
The $K$-finite matrix coefficients $m_{v,v'}$ of an admissible representation of $G_\R$ have an asymptotic expansion on the negative Weyl chamber $A_\R^-$ of expressions of the form~\ref{Deligne asymptotics}, where the complex characters $\lambda_i$ of $A_\R$ and powers of logarithms $m_i$ appear in the generalized eigenvalue decomposition of the $\fa$-action on $\fn$-coinvariants.\end{theorem}

This recovers in particular the result of~\cite{CasselmanAsymptoticbehaviormatrix1982} that the leading exponents of $V$ (with respect to dominance order) appear as weights of the $A$-action on the $\mathfrak n$-homology $H_0(\mathfrak n, V)$ -- a special case of a theorem of Mili\v{c}i\'{c}~\cite{MilicicAsymptoticbehaviourmatrix1977} that the leading exponents are precisely the minimal weights of the $A$-action on $\mathfrak n$-homology, for $\fn=\Lie(N)$ associated to the radical of the minimal parabolic $P_\R$. As explained in~\cite{CasselmanAsymptoticbehaviormatrix1982}, one can then deduce Harish-Chandra's results that the crucial growth properties (temperedness and $p$-integrability modulo center) of the matrix coefficient -- and, varying $v,v'$, of the representation $V$ itself -- are completely controlled by the leading exponent.

More generally, Theorem~\ref{thm:parabolic} can be used to study ``asymptotics along a wall'' as in~\cite{CasselmanAsymptoticbehaviormatrix1982} -- the asymptotic behavior of matrix coefficients in other directions is controlled by the matrix coefficients of other parabolic restrictions of $V$ (i.e., along other strata in $\Gbar$).

\bibliography{BZ-Ganev.bib}

\end{document}